\documentclass[12pt,twoside=semi,leqno,abstracton
]{scrartcl}
\pdfoutput=1%
\setkomafont{title}{\LARGE\bfseries}
\setkomafont{subtitle}{\itshape}
\setkomafont{disposition}{\normalfont\normalcolor\bfseries}
\setkomafont{section}{\large\bfseries\boldmath}
\setkomafont{subsection}{\large\bfseries\itshape\boldmath}
\setkomafont{dictum}{\scshape}
\setkomafont{paragraph}{\bfseries\boldmath}
\setkomafont{pagehead}{\normalfont\upshape\small}
\setkomafont{pagenumber}{\normalfont\upshape\normalsize}
\usepackage[arxiv%
]{MyKoma}
\usepackage{tikz}

\makeatletter

 \usepackage{chngcntr}%
 \counterwithout{equation}{section}

 \makeatother

\begin{document}

\title{A Robust Version of Convex Integral Functionals }

\author{Keita Owari\thanks{The author gratefully acknowledges
    the financial support of the Center of Advanced Research in
    Finance (CARF) at the Graduate School of Economics, The University
    of Tokyo.}}
\def\shorttitle{Robust Integral Functionals}%
\date{\footnotesize First Version: \printdate{26/05/2013}\\ Major Revision: \printdate{25/06/2014}\\ This version: \printdate{21/05/2015}}%

\def\shortauthor{K. Owari}
\address{Graduate School of Economics, The University of Tokyo \newline
  7-3-1 Hongo, Bunkyo-ku, Tokyo 113-0033, Japan}

\email{owari@e.u-tokyo.ac.jp}

\abstract{We study the pointwise supremum of convex integral
  functionals
  \begin{align*}
    \mathcal{I}_{f,\gamma}(\xi)=\sup_{Q} \left( \int_\Omega
    f(\omega,\xi(\omega))Q(d\omega)-\gamma(Q)\right),\quad \xi\in L^\infty(\Omega,\mathcal{F},\mathbb{P}),
  \end{align*}
  where $f:\Omega\times\mathbb{R}\rightarrow\overline{\mathbb{R}}$ is
  a proper normal convex integrand, $\gamma$ is a proper convex
  function on the set of probability measures absolutely continuous
  w.r.t. $\mathbb{P}$, and the supremum is taken over all such
  measures.  We give a pair of upper and lower bounds for the
  conjugate of $\mathcal{I}_{f,\gamma}$ as direct sums of a common
  regular part and respective singular parts; they coincide when
  $\mathrm{dom}(\gamma)=\{\mathbb{P}\}$ as Rockafellar's classical
  result, while both inequalities can generally be strict.  We then
  investigate when the conjugate eliminates the singular measures,
  which a fortiori yields the equality in bounds, and its relation to
  other finer regularity properties of the original functional and of
  the conjugate.  }

\keyWords{convex integral functionals, duality, robust stochastic
  optimization, financial mathematics}%
\maketitle

\section{Introduction}
\label{sec:Intro}

Let $(\Omega,\mathcal{F},\mathbb{P})$ be a probability space and
$f:\Omega\times\mathbb{R}\rightarrow (-\infty,\infty]$ a measurable mapping
with $f(\omega,\cdot)$ being proper, convex, lsc for
a.e. $\omega$. Then $\xi\mapsto I_{f}(\xi):= \int_\Omega
f(\omega,\xi(\omega))\mathbb{P}(d\omega)$ defines a convex functional on
$L^\infty:=L^\infty(\Omega,\mathcal{F},\mathbb{P})$, called the \emph{convex integral
  functional}. Among many others, R.T.~Rockafellar obtained in
\citep{MR0310612} that under mild integrability assumptions on $f$,
the conjugate $I_f^*(\nu)=\sup_{\xi\in L^\infty}(\nu(\xi)-I_f(\xi))$
of $I_f$ is expressed as the direct sum of regular and singular parts
(which we call the \emph{Rockafellar theorem}):
\begin{equation}
  \label{eq:RockafellarClassical1}
  I_f^*(\nu)=I_{f^*}(d\nu_r/d\mathbb{P})+\sup_{\xi\in \mathrm{dom}(I_f)}\nu_s(\xi),
  \quad\forall \nu\in (L^\infty)^*,
\end{equation}
where $f^*(\omega,y):=\sup_x(xy-f(\omega,x))$ and $\nu_r$
(resp. $\nu_s$) denotes the regular (resp. singular) part of
$\nu\in(L^\infty)^*$ in the Yosida-Hewitt decomposition.  In
particular, if $I_f$ is finite everywhere on $L^\infty$, the conjugate
$I_f^*$ \emph{``eliminates'' the singular elements} of $(L^\infty)^*$
in that $I_f^*(\nu)=\infty$ unless $\nu$ is $\sigma$-additive. The
latter property implies the weak-compactness of all the sublevel sets
of $I_f^*|_{L^1}$ and the continuity of $I_f$ for the Mackey topology
$\tau(L^\infty,L^1)$ and so on (e.g. \cite[][Th.~3K]{MR0512209}).

This paper is concerned with a \emph{robust version} of integral
functionals of the form
\begin{equation*}
  \mathcal{I}_{f,\gamma}(\xi) 
  :=\sup_{Q\in\mathcal{Q}}\left(\int_\Omega f(\omega,\xi(\omega))Q(d\omega)-\gamma(Q)\right), \quad
  \xi\in L^\infty. 
\end{equation*}
where $\mathcal{Q}$ is the set of all probability measures $Q$ absolutely
continuous w.r.t. $\mathbb{P}$, and $\gamma$ is a convex penalty function on
$\mathcal{Q}$. (see Section~\ref{sec:RobIntBasic} for precise formulation).
As the pointwise supremum of convex functions, $\mathcal{I}_{f,\gamma}$ is
convex, lower semicontinuous if so is each $I_{f,Q}(\xi):=\int_\Omega
f(\omega,\xi(\omega))Q(d\omega)$, and is even norm-continuous as soon
as it is finite everywhere. On the other hand, it is less obvious what
the convex conjugate $\mathcal{I}_{f,\gamma}^*$ is, when singular measures are
eliminated, and whether the latter property yields finer regularity
properties of $\mathcal{I}_{f,\gamma}$ and $\mathcal{I}_{f,\gamma}^*$.

Our main result (Theorem~\ref{thm:RobustRockafellar}) is a robust
version of the Rockafellar theorem which consists of a pair of upper
and lower bounds (instead of a single equality) for $\mathcal{I}_{f,\gamma}^*$
on $(L^\infty)^*$. Both bounds are of forms analogous to
(\ref{eq:RockafellarClassical1}) with a common regular part, but a
difference appears in the singular parts. They coincide in the
classical case $I_f=\mathcal{I}_{f,\delta_{\{\mathbb{P}\}}}$, while the subsequent
Example~\ref{ex:CounterEx1} shows that both inequalities can be strict
and one may not hope for sharper bounds in general.  The same example
shows also that the everywhere finiteness of $\mathcal{I}_{f,\gamma}$ is not
enough for the property that $\mathcal{I}^*_{f,\gamma}$ \emph{eliminates the
  singular measures}, while the lower bound in
Theorem~\ref{thm:RobustRockafellar} provides us a simple sufficient
condition in a form of ``uniform integrability''.  In
Theorem~\ref{thm:LebMaxCompact} and its corollaries, it is shown that
given $\mathrm{dom}(\mathcal{I}_{f,\gamma})=L^\infty$ (plus a technical assumption),
the condition is even necessary, and is equivalent to other finer
properties of $\mathcal{I}_{f,\gamma}$ and $\mathcal{I}_{f,\gamma}^*$, including the
weak compactness of sublevels of $\mathcal{I}_{f,\mathcal{P}}^*|_{L^1}$, the Mackey
continuity of $\mathcal{I}_{f,\gamma}$ on $L^\infty$ among others, which are
guaranteed solely by the finiteness in the classical case.

Certain class of \emph{robust optimization} problems are formulated as
(or reduced to) the minimization of a robust integral functional
$\mathcal{I}_{f,\gamma}$ over a convex set $\mathcal{C}\subset L^\infty$. Our initial
motivation of this work lies in the Fenchel duality for this type of
problems with $\langle L^\infty,L^1\rangle$ pairing:
\begin{equation}
  \label{eq:FenchelIntro}
  \inf_{\xi\in \mathcal{C}}\mathcal{I}_{f,\gamma}(\xi)=-\min_{\eta\in L^1}\Bigl(\mathcal{I}_{f,\gamma}^*(\eta)
    +\sup_{\xi\in \mathcal{C}}\langle -\xi,\eta\rangle\Bigr),
\end{equation}
The Rockafellar-type result provides us the precise form of
$\mathcal{I}_{f,\gamma}^*|_{L^1}$, hence of the \emph{dual problem}, while the
classical Fenchel duality theorem tells us that a sufficient condition
for (\ref{eq:FenchelIntro}) is the $\tau(L^\infty,L^1)$-continuity of
$\mathcal{I}_{f,\gamma}$ at some point $\xi_0\in \mathcal{C}\cap
\mathrm{dom}(\mathcal{I}_{f,\gamma})$, or in another view, the finiteness implies (via
the norm continuity) the $\langle
L^\infty,(L^\infty)^*\rangle$-duality which reduces to the $\langle
L^\infty,L^1\rangle$-duality if $\mathcal{I}_{f,\gamma}^*$ eliminates the
singular elements of $(L^\infty)^*$.

Our motivating example of minimization of $\mathcal{I}_{f,\gamma}$ is the
\emph{robust utility maximization}
\begin{align}
  \label{eq:RobUtilIntro1}
  \text{maximize}\quad u(\xi):=\inf_{Q\in\mathcal{Q}}
  \left(\mathbb{E}_Q[U(\cdot,\xi)] +\gamma(Q)\right) \quad\text{over a convex
    cone } \mathcal{C}\subset L^\infty
\end{align}
where $U:\Omega\times\mathbb{R}\rightarrow\mathbb{R}$ is a \emph{random utility
  function}. In this context, each $Q\in\mathcal{Q}$ is considered as a model
used to evaluate the quality of control $\xi$ via the expected utility
$\mathbb{E}_Q[U(\cdot, \xi)]=\int_\Omega U(\omega,\xi(\omega))Q(d\omega)$,
but we are not sure which model is true; so we optimize the worst case
among all models $Q$ penalized by $\gamma(Q)$ according to the
likelihood.  Note that (\ref{eq:RobUtilIntro1}) is equivalent to
minimize $\mathcal{I}_{f_U,\gamma}$ with $f_U(\omega,x)=-U(\omega,-x)$ over
the cone $-\mathcal{C}$. The corresponding Fenchel duality in $\langle
L^\infty,L^1\rangle$ constitute a half of what we call the
\emph{martingale duality} in mathematical finance; given the $\langle
L^\infty,L^1\rangle$-duality with a ``good'' cone $\mathcal{C}$ having the
polar generated by so-called \emph{local martingale measures}, the
theory of (semi)martingales takes care of the other half.

\section{Preliminaries}
\label{sec:RobIntBasic}

We use the probabilistic notation.  Let
$(\Omega,\mathcal{F},\mathbb{P})$ be a \emph{complete} probability
space and $L^0:=L^0(\Omega,\mathcal{F},\mathbb{P})$ denote the space
of (equivalence classes modulo $\mathbb{P}$-a.s. equality of)
$\mathbb{R}$-valued random variables defined on it. As usual, we do
not distinguish a random variable and the class it generates, and a
constant $c\in\mathbb{R}$ is regarded as the random variable
$c\ind_\Omega$. Here $\ind_A$ denotes the indicator of a set $A$ in
measure theoretic sense while $\delta_A :=\infty\ind_{A^c}$ is the one
in convex analysis. The $\mathbb{P}$-expectation of $\xi\in L^0$ is
denoted by $\mathbb{E}[\xi]:=\int_\Omega \xi(\omega)
\mathbb{P}(d\omega)$ and we write $L^p
:=L^p(\Omega,\mathcal{F},\mathbb{P})$ and
$\|\cdot\|_p:=\|\cdot\|_{L^p}$ for each $1\leq p\leq \infty$.  Any
probabilistic notation without reference to a measure is to be
understood with respect to $\mathbb{P}$. Especially, ``a.s.''  means
``$\mathbb{P}$-a.s.'', and \emph{identification of random variables is
  always made by $\mathbb{P}$}.  For other probability measures $Q$
absolutely continuous with respect to $\mathbb{P}$ ($Q\ll
\mathbb{P}$), we write $\mathbb{E}_Q[\cdot]$ for $Q$-expectation,
$L^p(Q):=\{\xi\in L^0:\, \mathbb{E}_Q[|\xi|^p]<\infty\}$ (which is a
set of $\mathbb{P}$-classes!) etc, explicitly indicating the measure
that involves.  We write $Q\sim P$ to mean $Q$ and $P$ are equivalent
($Q\ll P$ and $P\ll Q$).

The norm dual of $L^\infty$ is $ba:=ba(\Omega,\mathcal{F},\mathbb{P})$, the space of
all \emph{bounded finitely additive signed measures} $\nu$ respecting
$\mathbb{P}$-null sets, i.e., $\sup_{A\in\mathcal{F}}|\nu(A)|<\infty$, $\nu(A\cup
B)=\nu(A)+\nu(B)$ if $A\cap B=\emptyset$, and $\nu(A)=0$ if $\mathbb{P}(A)=0$
(see \cite[pp.~354-357]{hewitt_stromberg75:_real}). The bilinear form
of $(L^\infty,ba)$ is given by the (Radon) integral
$\nu(\xi)=\int_\Omega\xi d\nu$ which coincides with the usual integral
when $\nu$ is $\sigma$-additive. We regard any
\emph{$\sigma$-additive} $\nu\in ba$ as an element of $L^1$ via the
mapping $\nu\mapsto d\nu/d\mathbb{P}$ which is an isometry from the subspace
of such $\nu$'s onto $L^1$. In particular, the set 
\begin{align*}
  \mathcal{Q}:=\{Q:\, \text{probability measures on $(\Omega,\mathcal{F})$ with
    $Q\ll\mathbb{P}$}\}
\end{align*}
is regarded as $\{\eta\in L^1:\, \eta\geq 0,\,\mathbb{E}[\eta]=1\}$.  A $\nu\in
ba$ is said to be \emph{purely finitely additive} if there exists a
sequence $(A_n)$ in $\mathcal{F}$ such that $\mathbb{P}(A_n)\nearrow 1$ but
$|\nu|(A_n)=0$ for all $n$, and we denote by $ba^s$ the totality of
such $\nu\in ba$. Any $\nu \in ba$ admits a unique \emph{Yosida-Hewitt
  decomposition} $\nu=\nu_r+\nu_s$ where $\nu_r$ is the \emph{regular}
($\sigma$-additive) part, and $\nu_s$ is the \emph{purely finitely
  additive part}
(e.g. \citep[Th.~III.7.8]{dunford58:_linear_operat_I}), thus
$ba=L^1\oplus ba^s$ with the above identification.  We denote by
$ba_+$ the set of \emph{positive} elements of $ba$ and $ba_+^s:=
ba_+\cap ba^s$ etc.

\subsection{Convex Penalty Function and Associated Orlicz Spaces}
\label{sec:ConvPenal}

We make the following assumption on the \emph{penalty function}
$\gamma$:
\begin{assumption}
  \label{as:PenaltyFunc}
  $\gamma:\mathcal{Q}\rightarrow \overline{\mathbb{R}}_+$ is a
  $\sigma(L^1,L^\infty)$-lsc proper convex function such that
\begin{align}
  \label{eq:PenaltyFuncAssump1}
  & \inf_{Q\in\mathcal{Q}}\gamma(Q)=0;\\
  \label{eq:PenaltyFuncSensitive}
  &\exists Q_0\in \mathcal{Q} \text{ with } Q_0\sim\mathbb{P}\text{
    and }  \gamma(Q_0)<\infty;\\
  \label{eq:PenaltyFuncCompact}
  &\{Q\in\mathcal{Q}:\, \gamma(Q)\leq 1\} \text{ is
    $\sigma(L^1,L^\infty)$-compact}\\\notag
&\qquad\qquad (\Leftrightarrow\,\text{closed
    and uniformly integrable}).
\end{align}
We denote $\mathcal{Q}_\gamma:=\{Q\in\mathcal{Q}:\, \gamma(Q)<\infty\}$, the effective
domain of $\gamma$.
\end{assumption}

\begin{remark}
  \label{rem:PenaltyFuncAssump}
  (\ref{eq:PenaltyFuncAssump1}) and (\ref{eq:PenaltyFuncSensitive})
  are normalizing assumptions and the latter one is equivalent to
  saying that $\mathcal{Q}_\gamma\sim \mathbb{P}$, i.e., for each
  $A\in\mathcal{F}$, $\mathbb{P}(A)=0$ iff $Q(A)=0$ for every
  $Q\in\mathcal{Q}_{\gamma}$.  Given the lower semicontinuity of
  $\gamma$, (\ref{eq:PenaltyFuncCompact}) implies that the infimum in
  (\ref{eq:PenaltyFuncAssump1}) is attained, and
  (\ref{eq:PenaltyFuncCompact}) is equivalent to apparently stronger
  \begin{equation*}
    \{Q\in\mathcal{Q}:\,\gamma(Q)\leq c\}\text{ is $\sigma(L^1,L^\infty)$-compact, }\forall c>0.
  \end{equation*}
  Indeed, for any $c>1$, pick a $Q_c\in\mathcal{Q}_\gamma$ with
  $\gamma(Q_c)<1/c$ (by (\ref{eq:PenaltyFuncAssump1})), then
  $\gamma(Q)\leq c$ implies
  $\gamma\left(\frac1{1+c}Q+\frac{c}{1+c}Q_c\right)\leq
  \frac1{1+c}\cdot c+\frac{c}{1+c}\cdot \frac1c=1$, thus
  (\ref{eq:PenaltyFuncCompact}) implies the weak compactness of
  $\left\{\frac1{1+c}Q+\frac{c}{1+c}Q_c:\,\gamma(Q)\leq c\right\}$,
  hence of $\{Q\in \mathcal{Q}_\gamma:\, \gamma(Q)\leq c\}$.
\end{remark}

In general, any lower semicontinuous proper convex function $\gamma$
on $\mathcal{Q}$ defines a function
\begin{align*}
  \rho_\gamma(\xi)
  :=\sup_{Q\in\mathcal{Q}_\gamma}\left(\mathbb{E}_Q[\xi]-\gamma(Q)\right)\quad
  \text{on }\left\{\xi\in L^0:\, \xi^-\in
    \textstyle\bigcap_{Q\in\mathcal{Q}_\gamma} L^1(Q)\right\}\supset
  L^\infty\cup L^0_+.
\end{align*}
Regardless of Assumption~\ref{as:PenaltyFunc}, $\rho_\gamma$
restricted to $L^\infty$ is a $\sigma(L^\infty,L^1)$-lsc finite-valued
monotone convex function with $\rho_\gamma(\xi+c)=\rho_\gamma(\xi)+c$
if $c$ is a constant, whose conjugate on $L^1$ is
$\gamma(\eta)\ind_{\mathcal{Q}}(\eta)+\infty\ind_{L^1\setminus
  \mathcal{Q}}(\eta)$. In financial mathematics, such a function is
called a convex risk measure (up to a change of sign). Then
(\ref{eq:PenaltyFuncAssump1}) reads as $\rho_\gamma(0)=0$,
(\ref{eq:PenaltyFuncSensitive}) as $\rho_\gamma(\varepsilon\ind_A)=0$
for some $\varepsilon>0$ $\Rightarrow$ $\mathbb{P}(A)=0$, while
(\ref{eq:PenaltyFuncCompact}) is equivalent to saying that
$\rho_\gamma$ has the \emph{Lebesgue property on $L^\infty$} (see
\citep{jouini06:_law_fatou,MR2648597,MR2924150}):
\begin{align*}
  \sup_n\|\xi_n\|_\infty<\infty,\,\xi_n\rightarrow\xi\text{ a.s. }
  \Rightarrow\, \rho_\gamma(\xi)=\lim_n\rho_\gamma(\xi_n).
\end{align*}

To a penalty function $\gamma$ satisfying
Assumption~\ref{as:PenaltyFunc}, we associate the gauge norm
\begin{equation*}
  \|\xi\|_{\rho_\gamma} =\inf\left\{\lambda>0:\,
    \rho_\gamma(|\xi|/\lambda)\leq 1\right\}=\sup_{Q\in\mathcal{Q}_\gamma}
  \frac{\mathbb{E}_Q[|\xi|]}{1+\gamma(Q)}, \quad \xi\in L^0.
\end{equation*}
In view of (\ref{eq:PenaltyFuncSensitive}), this is indeed a norm on
the \emph{Orlicz space}
\begin{align*}
  L^{\rho_\gamma} &=\left\{\xi\in L^0:\, \exists \alpha>0\text{ with }
    \rho_\gamma(\alpha|\xi|)<\infty\right\} =\{\xi\in L^0:\,
  \|\xi\|_{\rho_\gamma}<\infty\},
\end{align*}
which is a \emph{solid} subspace (lattice ideal) of $L^0$ (i.e.,
$\xi\in L^{\rho_\gamma}$ and $|\zeta|\leq |\xi|$ a.s. $\Rightarrow$
$\zeta\in L^{\rho_\gamma}$), and
$(L^{\rho_\gamma},\|\cdot\|_{\rho_\gamma})$ is a Banach lattice. We
consider the following subspaces of $L^{\rho_\gamma}$ too:
\begin{align*}
  M^{\rho_\gamma}&:=\left\{\xi\in L^0:\, \forall \alpha>0,\, \rho_\gamma(\alpha|\xi|)<\infty\right\},\\
  M^{\rho_\gamma}_u&:=\left\{\xi\in L^0:\, \forall
    \alpha>0,\,\lim_N\rho_\gamma\left(\alpha|\xi|\ind_{\{|\xi|>N\}}\right)=0\right\}.
\end{align*}
Both $M^{\rho_\gamma}$ and $M^{\rho_\gamma}_u$ are solid as subspaces
of $L^0$. These Orlicz-type spaces are studied in \citep{MR3165235}
where $\gamma$ corresponds to $\varphi_0^*$ (more precisely
$\varphi^*_0=\gamma\ind_\mathcal{Q}+\infty\ind_{L^1\setminus \mathcal{Q}}$) and
$\hat\varphi=\rho_\gamma$ in the notation of the current paper. In
general under Assumption~\ref{as:PenaltyFunc},
\begin{equation*}
  L^\infty\subset  M^{\rho_\gamma}_u 
  \subset M^{\rho_\gamma} 
  \subset L^{\rho_\gamma}\subset\textstyle\bigcap_{Q\in\mathcal{Q}_\gamma}L^1(Q),
\end{equation*}
while all inclusions can generally be strict (\citep[][Examples~3.6
and 3.7]{MR3165235}). From the last inclusion, $\rho_\gamma$ is
well-defined on $L^{\rho_\gamma}$ as a proper monotone convex
function, and it is finite-valued on $M^{\rho_\gamma}$ while not on
$L^{\rho_\gamma}$ (unless $M^{\rho_\gamma}=L^{\rho_\gamma}$), and
$M^{\rho_\gamma}_u$ is the maximum solid subspace of $L^0$ to which
$\rho_\gamma$ retain the Lebesgue property (as the order-continuity;
see \citep[Theorem~3.5]{MR3165235}), and is characterized by a uniform
integrability property \citep[Theorem~3.8]{MR3165235}: for $\xi\in
M^{\rho_\gamma}$,
\begin{equation}
  \label{eq:UIOrliczIFF}
  \xi\in M^{\rho_\gamma}_u\,
  \Leftrightarrow\, \{\xi dQ/d\mathbb{P}:\, \gamma(Q)\leq c\} \text{ is uniformly integrable, }\forall c>0.
\end{equation}
Actually, by the same argument as in
Remark~\ref{rem:PenaltyFuncAssump}, we have only to consider the case
$c=1$.

\subsection{Robust Version of Integral Functionals}

In the sequel, let $f:\Omega\times\mathbb{R}\rightarrow\overline{\mathbb{R}}$ be a
\emph{proper normal convex integrand} on $\mathbb{R}$, i.e.,
\begin{equation}
  \label{eq:NormalConvInt}
  \begin{split}
    &f\text{ is $\mathcal{F}\otimes\mathcal{B}(\mathbb{R})$-measurable
      and
    }\\
    &\text{$f(\omega,\cdot)$ is an lsc proper convex function for
      a.e. }\omega.
  \end{split}
\end{equation}
Since $\mathcal{F}$ is assumed $\mathbb{P}$-complete, the first part is equivalent to
the measurability of epigraphical mapping (see
\citep[Ch.~14]{rockafellar_wets98} for a general reference). As
immediate consequences of (\ref{eq:NormalConvInt}), $f(\cdot,\xi)$ is
$\mathcal{F}$-measurable for each $\xi\in L^0$, and $f^*$ is also a proper
normal convex integrand where
\begin{equation*}
  \label{eq:OutlineConjNormal1}
  f^*(\omega,y):=\sup_{x\in\mathbb{R}}(xy-f(\omega,x)),\quad \omega\in \Omega,\, y\in \mathbb{R},
\end{equation*}

Given such $f$ and $\gamma$, we define a robust analogue of convex
integral functional
\begin{equation}
  \label{eq:SupIntFunc1}
  \mathcal{I}_{f,\gamma}(\xi):=\sup_{Q\in\mathcal{Q}_\gamma} 
  \left(\mathbb{E}_Q[f(\cdot,\xi)]-\gamma(Q)\right) =\rho_\gamma\left(f(\cdot,\xi)\right),
  \quad \xi\in L^\infty.
\end{equation}
This functional is well-defined as a proper convex functional on
$L^\infty$ as soon as:
\begin{align}
  \label{eq:IntegRobust1}
  &\exists \xi_0\in L^\infty\text{ s.t. } f(\cdot,\xi_0)^+\in M^{\rho_\gamma},\\
  \label{eq:AsConjInteg1}
  &\exists \eta_0\in L^{\rho_\gamma} \text{ s.t. }
  f^*(\cdot,\eta_0)^+\in L^{\rho_\gamma}.
\end{align}
Indeed, since $f(\cdot,\xi)\geq \xi\eta_0-f^*(\cdot,\eta_0)^+$,
(\ref{eq:AsConjInteg1}) implies $f(\cdot,\xi)^-\in L^{\rho_\gamma}
\subset \bigcap_{Q\in\mathcal{Q}_\gamma} L^1(Q)$, so
$\mathbb{E}_Q[f(\cdot,\xi)]$ is well-defined with values in
$(-\infty,\infty]$ for each $Q\in\mathcal{Q}_\gamma$ and $\xi\in
L^\infty$, thus so is
$\mathcal{I}_{f,\gamma}(\xi)=\sup_{Q\in\mathcal{Q}_\gamma}\left(\mathbb{E}_Q[f(\cdot,
  \xi)]-\gamma(Q)\right)$, while (\ref{eq:IntegRobust1}) shows that
$\mathcal{I}_{f,\gamma}(\xi_0)<\infty$, thus
$\mathcal{I}_{f,\gamma}\not\equiv\infty$. Also,
$\mathcal{I}_{f,\gamma}$ is convex as the pointwise supremum of convex
functions. Note also that
\begin{align}\label{eq:IFDom}
  \{\xi\in L^\infty:\, f(\cdot, \xi)^+\in M^{\rho_\gamma}\}\subset
  \mathrm{dom}(\mathcal{I}_{f,\gamma}) \subset \{\xi\in L^\infty:\, f(\cdot,\xi)^+\in
  L^{\rho_\gamma}\}.
\end{align}
where $\mathrm{dom}( \mathcal{I}_{f,\gamma}) :=\{\xi \in L^\infty:\,
\mathcal{I}_{f,\gamma}(\xi)<\infty\}$. In particular, by
(\ref{eq:PenaltyFuncSensitive}),
\begin{equation}
  \label{eq:DomIFImplyDomF}
  \xi\in \mathrm{dom}(\mathcal{I}_{f,\gamma})\,\Rightarrow\, \xi(\omega)\in \mathrm{dom} f(\omega,\cdot)\text{ for a.e. }\omega.
\end{equation}
If $M^{\rho_\gamma}=L^{\rho_\gamma}$, all three sets in
(\ref{eq:IFDom}) coincide, so $\rho_\gamma\left(f(\cdot,
  \xi)^+\right)<\infty$ as soon as $\mathcal{I}_{f,\gamma}(\xi)<\infty$. In
general, however, $\rho_\gamma\left(f(\cdot,\xi)^+\right)=\infty$ may
happen even if $\xi\in \mathrm{dom}(\mathcal{I}_{f,\gamma})$.

We next check that $\mathcal{I}_{f,\gamma}$ has a nice regularity on
$L^\infty$.
\begin{lemma}
  \label{lem:WellDefFatou}
  Under Assumption~\ref{as:PenaltyFunc}, (\ref{eq:IntegRobust1}) and
  (\ref{eq:AsConjInteg1}), $\mathcal{I}_{f,\gamma}$ is
  $\sigma(L^\infty,L^1)$-lower semicontinuous, or equivalently
  $\mathcal{I}_{f,\gamma}$ has the Fatou property:
  \begin{equation}
    \label{eq:Fatou1}
    \sup_n\|\xi_n\|_\infty<\infty,\,\xi_n\rightarrow \xi\text{ a.s. } 
    \Rightarrow\, \mathcal{I}_{f,\gamma}(\xi)\leq \liminf_n\mathcal{I}_f(\xi_n).
  \end{equation}

\end{lemma}
\begin{proof}
  Suppose $a:=\sup_n\|\xi_n\|_\infty<\infty$ and $\xi_n\rightarrow
  \xi$ a.s., then automatically $\xi\in L^\infty$. With $\eta_0$ as in
  (\ref{eq:AsConjInteg1}), $f(\cdot, \xi_n)\geq
  -a|\eta_0|-f^*(\cdot,\eta_0)^+\in
  L^{\rho_\gamma}\subset\bigcap_{Q\in\mathcal{Q}_\gamma}L^1(Q)$. Hence for each
  $Q\in\mathcal{Q}_\gamma$, Fatou's lemma shows that $\mathbb{E}_Q[f(\cdot,\xi)]\leq
  \liminf_n\mathbb{E}_Q[f(\cdot,\xi_n)]$, and we deduce (\ref{eq:Fatou1}) as
  \begin{align*}
    \sup_{Q\in\mathcal{Q}_\gamma}\left(\mathbb{E}_Q[f(\cdot,\xi)]-\gamma(Q)\right)
    &\leq \sup_{Q\in\mathcal{Q}_\gamma}\left(
      \liminf_n\mathbb{E}_Q[f(\cdot,\xi_n)]-\gamma(Q)\right) \\
    &\leq
    \liminf_n\sup_{Q\in\mathcal{Q}_\gamma}\left(\mathbb{E}_Q[f(\cdot,\xi_n)]-\gamma(Q)\right).
  \end{align*}
  The equivalence between (\ref{eq:Fatou1}) and the weak*-lower
  semicontinuity follows from a well-known consequence of the
  Krein-Šmulian and Mackey-Arens theorems that (see
  e.g. \citep{MR0372565}) a \emph{convex} set $C\subset L^\infty$ is
  weak*-closed if and only if $C\cap \{\xi:\,\|\xi\|_\infty\leq a\}$
  is $L^0$-closed for all $a>0$.
\end{proof}

\subsection{Robust $f^*$-divergence}
\label{sec:Divergence}

We proceed to the functional that plays the role of $I_{f^*}$ in the
classical case. Let
\begin{equation*}
  \tilde f^*(\omega,y,z):=\sup_{x\in \mathrm{dom}f(\omega,\cdot)}(xy-zf(\omega,x)),
  \quad \omega\in\Omega,\, (y,z)\in \mathbb{R}\times\mathbb{R}_+.
\end{equation*}
Noting that $(a^-_{f^*},a^+_{f^*})\subset \mathrm{dom}f\subset
[a^-_{f^*},a^+_{f^*}]$ where $a^{\pm}_{f^*}
:=\lim_{k\rightarrow\pm}f^*(\cdot,k)/k$, we have
\begin{align}\label{eq:ConjTilde3}
  \tilde f^*(\omega,y,z)=
  \begin{cases}
    0&\text{ if }y=z=0,\\
    y\cdot a^\pm_{f^*}(\omega)
    &\text{ if }y\gtrless0,z=0,\\
    zf^*(\omega, y/z)&\text{ if }z>0.
  \end{cases}
\end{align}
\begin{lemma}
  \label{lem:ftildeNormal1}
  Suppose (\ref{eq:NormalConvInt}). Then $\tilde
  f^*:\Omega\times\mathbb{R}\times\mathbb{R}_+\rightarrow \overline{\mathbb{R}}$ is a proper
  normal convex integrand on $\mathbb{R}\times\mathbb{R}_+$, i.e., it is $\mathcal{F}\otimes
  \mathcal{B}(\mathbb{R})\otimes\mathcal{B}(\mathbb{R}_+)$-measurable and $(y,z)\mapsto \tilde
  f^*(\omega,y,z)$ is a lower semicontinuous proper convex function
  for a.e. $\omega\in\Omega$. Also, for a.e. $\omega\in \Omega$,
\begin{equation}
  \label{eq:IneqYoungFtilde}
  xy\leq zf(\omega, x)+\tilde f^*(\omega,y,z),
  \quad\,\forall x\in\mathrm{dom}f(\omega,\cdot),\, \forall y \in \mathbb{R},\, \forall z\geq 0. 
\end{equation}

\end{lemma}
\begin{proof}
  Since $f$ is normal, there exists a sequence of measurable functions
  $(\xi_n)_{n\in\mathbb{N}}\subset L^0$ such that
  $\{\xi_n(\omega)\}_n\cap \mathrm{dom} f(\omega,\cdot)$ is dense in
  $\mathrm{dom} f(\omega,\cdot)$ (\citep{MR0512209},
  Proposition~2D). Modifying the sequence as
  $\bar\xi_n:=\xi_n\ind_{\{f(\cdot,\xi_n)<\infty\}}+\xi_0\ind_{\{f(\cdot,\xi_n)
    =\infty\}}$
  where $\xi_0\in \mathrm{dom}\mathcal{I}_{f,\gamma}$, we have for
  a.e. $\omega$, $\bar\xi_n(\omega)\in \mathrm{dom} f(\omega,\cdot)$
  and $\{\bar\xi_n(\omega)\}_n$ is dense in
  $\mathrm{dom} f(\omega,\cdot)$. Thus
  \begin{align*}
    \tilde f^*(\omega,y,z) =\sup_n(\bar\xi_n(\omega)y-
    zf(\cdot,\bar\xi_n(\omega))).
  \end{align*}
  Consequently, $\tilde f^*$ is a normal convex integrand as the
  countable supremum of affine integrands with $\tilde
  f^*(\cdot,0,0)=0$.  (\ref{eq:IneqYoungFtilde}) is obvious from the
  definition.
\end{proof}

Now we define $\mathcal{H}_{f^*}(\eta|Q):=\mathbb{E}\bigl[\tilde f^*\left(\cdot,
  \eta,dQ/ d\mathbb{P}\right)\bigr]$ for $\eta\in L^1$, $Q\in\mathcal{Q}_\gamma$, and
\begin{align}
  \label{eq:Entropy2}
  \mathcal{H}_{f^*,\gamma}(\eta)
  &:=\inf_{Q\in\mathcal{Q}_\gamma}\left(\mathcal{H}_{f^*}(\eta|Q)+\gamma(Q)\right),
  \quad \eta\in L^1.
\end{align}
In view of identification $\{\nu\in ba:\text{ $\sigma$-additive}\}=
L^1$, we define also for any $\sigma$-additive $\nu\in ba$,
\begin{align*}
  \mathcal{H}_{f^*}(\nu|Q) :=\mathcal{H}_{f^*}\left(d\nu/ d\mathbb{P}|Q\right),\quad
  \mathcal{H}_{f^*,\gamma}(\nu)=\mathcal{H}_{f^*,\gamma}(d\nu/d\mathbb{P}). 
\end{align*}
Since $\tilde f^*(\cdot,y,1)=f^*(\cdot,y)$, we recover
$\mathcal{H}_{f^*}(\eta|\mathbb{P})=\mathcal{H}_{f^*,\delta_{\{\mathbb{P}\}}}(\eta)=I_{f^*}(\eta)$
in the classical case.  Under the above assumptions,
$\mathcal{H}_{f^*}(\cdot|\cdot)$ and $\mathcal{H}_{f^*,\gamma}$ are well-defined.

\begin{lemma}
  \label{lem:SupIntWellDef}
  Under Assumption~\ref{as:PenaltyFunc}, (\ref{eq:NormalConvInt}),
  (\ref{eq:IntegRobust1}) and (\ref{eq:AsConjInteg1}),
  $\mathcal{H}_{f^*}(\cdot|\cdot)$ and $\mathcal{H}_{f^*,\gamma}$ are well-defined as
  proper convex functionals respectively on $L^1\times\mathcal{Q}_\gamma$ and
  $L^1$, and it holds
  \begin{equation}
    \label{eq:YoungIneqRob}
    \mathbb{E}[\eta\xi]\leq \mathcal{I}_{f,\gamma}(\xi)+\mathcal{H}_{f^*,\gamma}(\eta),\quad
    \forall \xi\in L^\infty,\,\forall\eta\in L^1.
  \end{equation}
\end{lemma}
\begin{proof}
  Let $\psi_Q:=dQ/d\mathbb{P}$ for each $Q\in\mathcal{Q}_\gamma$.  In view of
  (\ref{eq:DomIFImplyDomF}), we see that
  \begin{align}
    \label{eq:ProofPropEntropy1}
    \tilde f^*(\cdot, \eta, \psi_Q)\geq \xi\eta-\psi_Q f(\cdot, \xi)
    \geq \xi\eta-\psi_Q f(\cdot, \xi)^+\in L^1,
  \end{align}
  for any $\xi\in \mathrm{dom}(\mathcal{I}_{f,\gamma})$ ($\neq\emptyset$), $\eta\in L^1$
  and $Q\in\mathcal{Q}_\gamma$, thus $\mathcal{H}_{f^*}(\eta|Q)=\mathbb{E}\big[\tilde
  f^*(\cdot, \eta,\psi_Q)\big]>-\infty$ is well-defined, convex on
  $L^1\times \mathcal{Q}_\gamma$ (since $\tilde f^*$ is convex), and proper
  since $\mathcal{H}_{f^*}(\eta_0\psi_Q|Q)=\mathbb{E}_Q[f^*(\cdot, \eta_0)]$ with
  $\eta_0\in L^{\rho_\gamma}$ as in (\ref{eq:AsConjInteg1}) (then
  $\eta_0\psi_Q\in L^1$). Taking the expectation in
  (\ref{eq:ProofPropEntropy1}),
  \begin{align*}
    \mathcal{H}(\eta|Q)+\gamma(Q)\geq \mathbb{E}[\xi\eta]-(\mathbb{E}_Q[f(\cdot, \xi)]
    -\gamma(Q))\geq \mathbb{E}[\xi\eta]-\mathcal{I}_{f,\gamma}(\xi)>-\infty,
  \end{align*}
  for any $\xi\in \mathrm{dom}(\mathcal{I}_{f,\gamma})$, $\eta\in L^1$ and
  $Q\in\mathcal{Q}_\gamma$, so $\mathcal{H}_{f^*,\gamma,}(\eta)
  =\inf_{Q\in\mathcal{Q}_\gamma}\left(\mathcal{H}_{f^*}(\eta|Q)+\gamma(Q)\right)
  >-\infty$ and we have (\ref{eq:YoungIneqRob}) (which is trivially
  true when $\mathcal{I}_{f,\gamma}(\xi)=\infty$). The convexity of
  $\mathcal{H}_{f^*,\gamma}$ follows from that of
  $\mathcal{H}_{f^*}(\cdot|\cdot)+\gamma(\cdot)$ and of $\mathcal{Q}_\gamma$.
\end{proof}

\begin{remark}
  \label{rem:EntropyLSCEtc}
  Under assumption (\ref{eq:IntegrabilityD}) in the (main)
  Theorem~\ref{thm:RobustRockafellar} below, $\mathcal{H}_{f^*}+\gamma$
  (resp. $\mathcal{H}_{f^*,\gamma}$) is weakly lower semicontinuous on
  $L^1\times\mathcal{Q}_\gamma$ (resp.  $L^1$), and for each $\eta\in L^1$,
  the infimum $\inf_{Q\in\mathcal{Q}_\gamma}
  \left(\mathcal{H}_{f^*}(\eta|Q)+\gamma(Q)\right)$ is attained (see
  Appendix~\ref{sec:LSCEntropy}). Though we could give direct proofs
  here, we will not use these properties, and the lower semicontinuity
  of $\mathcal{H}_{f^*,\gamma}$ will be obtained as
  Corollary~\ref{cor:RobRockL1} to the main theorem which does not
  internally use that property.

  Note also that when $f$ (hence $f^*$ too) is non-random, all the
  integrability assumptions are trivialized, in which case
  $\mathcal{H}_{f^*}(\cdot|\cdot)$ is called the \emph{$f^*$-divergence} while
  $\mathcal{H}_{f^*,\gamma}$ is a slight generalization of \emph{robust
    $f^*$-divergence} (the latter is a special case with
  $\gamma(Q)=\delta_\mathcal{P}(Q)$ with $\mathcal{P}\subset\mathcal{Q}$), and the joint lower
  semicontinuity etc are found e.g. in
  \citep[][Lemma~2.7]{MR2247836}. In fact, if $f$ is finite
  ($\mathbb{P}(f(\cdot, x)<\infty,\,\forall x)=1$ $\Leftrightarrow$
  $\lim_{|y|\rightarrow\infty} f^*(\cdot, y)/|y|=\infty$ a.s.), we
  have from (\ref{eq:ConjTilde3}) that
\begin{align*}
  \label{eq:EntropyAlt1}
  \mathcal{H}_{f^*}(\nu|Q)=
  \begin{cases} 
    \mathbb{E}_Q[f^*(\cdot,d\nu/dQ)]&\text{if }\nu\ll Q,\\
    +\infty&\text{otherwise.}
  \end{cases}
\end{align*}
\end{remark}

Here are some typical examples of penalty function $\gamma$ and
associated integral functionals.
\begin{example}[Classical case]
  \label{ex:Classical}
  The ``classical'' integral functional
  $I_{f,\mathbb{P}}(\xi):=\mathbb{E}[f(\cdot,\xi)]$ corresponds to the penalty
  function $\gamma_\mathbb{P}(Q):=\delta_{\{\mathbb{P}\}}(Q)$ which clearly
  satisfies Assumption~\ref{as:PenaltyFunc}, and
  $\rho_{\gamma_\mathbb{P}}(\xi)=\mathbb{E}[\xi]$. Then
  $M^{\rho_\gamma}_u=M^{\rho_\gamma} =L^{\rho_\gamma}=L^1$, and the
  integrability assumptions (\ref{eq:IntegRobust1}) and
  (\ref{eq:AsConjInteg1}) are identical to the ones in
  \citep{MR0310612}:
  \begin{align*}
    \exists \xi_0\in L^\infty\text{ with } f(\cdot, \xi_0)^+\in
    L^1\quad\text{and}\quad \exists\eta_0\in L^1\text{ with }
    f^*(\cdot,\eta_0)^+\in L^1,
  \end{align*}
  under which (\ref{eq:RockafellarClassical1}) holds true
  (\citep{MR0310612}, Th.~1), and since
  $I_{f^*,\mathbb{P}}(\eta)=\mathcal{H}_{f^*}(\eta|\mathbb{P})$, it reads as
  \begin{align*}
    I_{f,\mathbb{P}}^*(\nu)=\mathcal{H}_{f^*}(\nu_r|\mathbb{P})+\sup_{\xi\in
      \mathrm{dom}(I_{f,\mathbb{P}})}\nu_s(\xi), \quad \forall\nu\in ba.
  \end{align*}
  We can also consider other probability $P\ll\mathbb{P}$ and
  $I_{f,P}:=I_{f,\gamma_P}$ where $\gamma_P=\delta_{\{P\}}$, but we
  need a little care when $P\not\sim\mathbb{P}$ (then
  (\ref{eq:PenaltyFuncSensitive}) is violated): here we consider
  $I_{f,P}$ as a functional on $L^\infty(\mathbb{P})$ rather than the space
  of $P$-equivalence classes of $P$-essentially bounded random
  variables. Equivalently, $I_{f,P}=I_{g,\mathbb{P}}$ with
  $g(\cdot,x)=f(\cdot,x)\frac{dP}{d\mathbb{P}}\ind_{\{dP/d\mathbb{P}>0\}}$, and its
  conjugate is
  \begin{equation}
    \label{eq:RockafellarNotEquivalent}
    I_{f,P}^*(\nu)=\mathcal{H}_{f^*}(\nu_r|P)+\infty\ind_{\{\nu_r\not\ll P\}} 
    +\sup_{\xi\in\mathrm{dom}(I_{f,P})}\nu_s(\xi).
  \end{equation}
\end{example}

\begin{example}[Homogeneous case]
  \label{ex:Coherent}
  The formulation (\ref{eq:SupIntFunc1}) covers the following form
  \begin{align*}
    \mathcal{I}_{f,\mathcal{P}}(\xi):= \sup_{Q\in\mathcal{P}}\mathbb{E}_Q[f(\cdot, \xi)]
    =\mathcal{I}_{f,\delta_\mathcal{P}}(\xi),
  \end{align*}
  where $\mathcal{P}\subset\mathcal{Q}$ is a nonempty convex set.
  $\delta_\mathcal{P}$ clearly satisfies
  (\ref{eq:PenaltyFuncAssump1}), and (\ref{eq:PenaltyFuncCompact})
  (resp. (\ref{eq:PenaltyFuncSensitive})) is equivalent to the weak
  compactness of $\mathcal{P}$ itself (resp. $\exists Q\in\mathcal{P}$
  with $Q\sim\mathbb{P}$), and
  $\rho_\mathcal{P}(\xi):=\rho_{\delta_\mathcal{P}}(\xi)=\sup_{Q\in\mathcal{P}}\mathbb{E}_Q[\xi]$
  is a \emph{positively homogeneous} monotone convex function, called
  \emph{sublinear expectation or coherent risk measure} (modulo change
  of sign). In particular, $M^{\rho_\gamma}=L^{\rho_\gamma}$, while
  $M^{\rho_\gamma}_u\subsetneq M^{\rho_\gamma}$ is possible (see
  \citep[][Example~3.7]{MR3165235}).
\end{example}

\begin{example}[Polyhedral case]
  \label{ex:Polyhedral}
  This is a special case of Example~\ref{ex:Coherent}. Suppose we are
  given a finite number of probability measures
  $P_1,...,P_n\ll\mathbb{P}$ which generate a polyhedral convex set
  $\mathcal{P}=\mathrm{conv}(P_1,...,P_n)$. This $\mathcal{P}$ is
  clearly (convex and) weakly compact in $L^1$,
  (\ref{eq:PenaltyFuncSensitive}) is equivalent to
  $\frac1n(P_1,...,P_n)\sim\mathbb{P}$, and we have
  $M^{\rho_\mathcal{P}}_u =M^{\rho_\mathcal{P}}
  =L^{\rho_\mathcal{P}}=\bigcap_{k\leq n} L^1(P_k)$ since
  \begin{align*}
    \frac1n\sum_{k\leq n}\mathbb{E}_{P_k}[|\xi|]\leq
    \|\xi\|_{\rho_{\mathcal{P}}} =\max_{1\leq k\leq n}
    \mathbb{E}_{P_k}[|\xi|] \leq \sum_{k\leq n}
    \mathbb{E}_{P_k}[|\xi|].
  \end{align*}
  In particular, noting that $I_{f,(\lambda_1 P_1 +\cdots
    +\lambda_nP_n)} =\lambda_1I_{f,P_1} +\cdots +\lambda I_{f,P_n}$,
  \begin{align*}
    \mathcal{I}_{f,\mathcal{P}}(\xi)
    =\sup_{Q\in\mathcal{P}}I_{f,Q}(\xi) =\max_{1\leq k\leq n}
    I_{f,P_k}(\xi), \quad \mathrm{dom}(\mathcal{I}_{f,\mathcal{P}})
    =\textstyle\bigcap_{k\leq n} \mathrm{dom}(I_{f,P_k}).
  \end{align*}
  In \citep[][Cor.~2.8.11]{MR1921556}, the conjugate of pointwise
  maximum of \emph{finitely many} convex functions is obtained, which
  reads in our context as (compare to (\ref{eq:RobRockPoly1}) below):
  if $\mathrm{dom}(\mathcal{I}_{f,\mathcal{P}})\neq \emptyset$,
  \begin{align}
    \label{eq:MaxFinitelyMany}
    \begin{split}
      \mathcal{I}_{f,\mathcal{P}}^*(\nu)%
      & =\min\left\{\varphi^*(\nu):\,
        \varphi\in\mathrm{conv}(I_{f,P_1},...,I_{f,P_n})\right\}\\
      &=\min_{Q\in\mathcal{P}}I_{f,Q}^*(\nu)
      \stackrel{\text{(\ref{eq:RockafellarNotEquivalent})}}
      =\min_{Q\in\mathcal{P}}\left\{\mathcal{H}_{f^*}
        (\nu_r|Q)+\infty\ind_{\{\nu_r\not\ll Q\}} +\sup_{\xi\in
          \mathrm{dom}(I_{f,Q})} \nu_s(\xi)\right\}.
    \end{split}
  \end{align}
\end{example}
\begin{example}[Entropic penalty]
  \label{ex:Entropic}
  Let $\gamma_\mathrm{ent}(Q)$ be the relative entropy of $Q$
  w.r.t. $\mathbb{P}$:
  \begin{equation*}
    \gamma_\mathrm{ent}(Q) 
    :=\mathcal{H}_{x\log x}(Q|\mathbb{P}) 
    :=\mathbb{E}\left[\frac{dQ}{d\mathbb{P}}\log \frac{dQ}{d\mathbb{P}}\right].
  \end{equation*}
  This function satisfies Assumption~\ref{as:PenaltyFunc}:
  $\gamma_\mathrm{ent}(\mathbb{P})=0$, hence
  (\ref{eq:PenaltyFuncAssump1}) and (\ref{eq:PenaltyFuncSensitive}),
  while (\ref{eq:PenaltyFuncCompact}) follows from the
  de~la~Vallée-Poussin theorem since
  $\lim_{x\rightarrow\infty}\frac{x\log x}x =\lim_{x\rightarrow
    \infty}\log x=\infty$. Let
  \begin{align*}
    L^{\Phi_\mathrm{\exp}} &:=\{\xi\in L^0:\, \exists\alpha>0,\,
    \mathbb{E}[\exp(\alpha |\xi|)]<\infty\}
    \quad \text{(exponential Orlicz space)},\\
    M^{\Phi_\mathrm{\exp}} &:=\{\xi\in L^0:\, \forall\alpha>0,\,
    \mathbb{E}[\exp(\alpha |\xi|)]<\infty\}\quad \text{(Morse
      subspace)}.
  \end{align*}
  In this case, $\rho_\mathrm{ent}(\xi)
  :=\rho_{\gamma_\mathrm{ent}}(\xi) =\log
  \mathbb{E}\left[\exp(\xi)\right]$ whenever $\xi^-\in
  L^{\Phi_{\exp}}$. Hence, $M^{\rho_\gamma}_u =M^{\rho_\gamma}
  =M^{\Phi_{\exp}}\subset L^{\Phi_{\exp}} =L^{\rho_\gamma}$ and
  $M^{\Phi_{\exp}} \subsetneq L^{\Phi_{\exp}}$ if
  $(\Omega,\mathcal{F},\mathbb{P})$ is atomless (e.g. exponential
  random variable).  The integrability assumptions
  (\ref{eq:IntegRobust1}) for $\xi_0\in L^\infty$ and
  (\ref{eq:AsConjInteg1}) for $\eta_0\in L^{\Phi_{\exp}}$ read as
  $\mathbb{E}\left[\exp(\alpha f(\cdot, \xi_0))\right]<\infty$ for all
  $\alpha>0$ and $\mathbb{E}\left[\exp(\varepsilon
    f^*(\cdot,\eta_0))\right]<\infty$ for some $\varepsilon>0$,
  respectively.  Moreover, the corresponding integral functional is
  explicitly written as
  \begin{align*}
    \mathcal{I}_{f,\gamma_{\mathrm{ent}}}(\xi)
    =\log\mathbb{E}\left[\exp\left(f(\cdot, \xi)\right)\right],
    \quad\forall \xi\in L^\infty.
  \end{align*}
\end{example}

\section{Statements of Main Results}
\label{sec:Results}

\subsection{A Rockafellar-Type Theorem for the Convex Conjugate}
\label{sec:Rockafellar}

The following is a \emph{robust analogue} of the Rockafellar theorem
for the conjugate of $\mathcal{I}_{f,\gamma}$
\begin{equation*}
  \mathcal{I}_{f,\gamma}^*(\nu) 
  :=\sup_{\xi\in L^\infty}(\nu(\xi)-\mathcal{I}_{f,\gamma}(\xi)), 
  \quad  \nu\in ba=ba(\Omega,\mathcal{F},\mathbb{P}).
\end{equation*}

\begin{theorem}
  \label{thm:RobustRockafellar}
  Suppose Assumption~\ref{as:PenaltyFunc}, (\ref{eq:NormalConvInt}),
  (\ref{eq:AsConjInteg1}) and
  \begin{align}
    \label{eq:IntegrabilityD}
    \exists \xi_0\in L^\infty \text{ such that } f(\cdot,\xi_0)^+\in
    M^{\rho_\gamma}_u.
  \end{align}
  Then for any $\nu\in ba$ with the Yosida-Hewitt decomposition
  $\nu=\nu_r+\nu_s$,
  \begin{equation}
    \label{eq:EstConj}
    \mathcal{H}_{f^*,\gamma}(\nu_r) +\sup_{\xi\in
      \mathcal{D}_{f,\gamma}}\nu_s(\xi) \leq \mathcal{I}_{f,\gamma}^*(\nu)
    \leq \mathcal{H}_{f^*,\gamma}(\nu_r)
    +\sup_{\xi\in\mathrm{dom}(\mathcal{I}_{f,\gamma})}\nu_s(\xi),
  \end{equation}
  where $\mathcal{D}_{f,\gamma} :=\{\xi\in L^\infty:\,
  f(\cdot,\xi)^+\in M^{\rho_\gamma}_u\} \subset
  \mathrm{dom}(\mathcal{I}_{f,\gamma})$ (by (\ref{eq:IFDom})).
\end{theorem}

A proof is given in Section~\ref{sec:ProofIntRepMain}. In contrast to
the classical Rockafellar theorem (\ref{eq:RockafellarClassical1}),
our robust version (\ref{eq:EstConj}) consists of two inequalities
instead of a single equality. But the possible difference appears only
in the singular part, thus
\begin{corollary}[Restriction to $L^1$]
  \label{cor:RobRockL1}
  Under the same assumptions as in
  Theorem~\ref{thm:RobustRockafellar},
  \begin{equation*}
    \mathcal{I}_{f,\gamma}^*(\eta) =\mathcal{H}_{f^*,\gamma}(\eta)
    =\inf_{Q\in\mathcal{Q}_\gamma}\left(\mathcal{H}_{f^*}(\eta|Q) 
      +\gamma(Q)\right),\quad \forall \eta\in L^1.
  \end{equation*}
  In particular, $\mathcal{H}_{f^*,\gamma}$ is weakly lower semicontinuous on
  $L^1$, and
  \begin{equation}
    \label{eq:DualRep1}
    \mathcal{I}_f(\xi) =\sup_{\eta\in L^1}\left(\mathbb{E}[\xi\eta]-\mathcal{H}_{f^*,\gamma}(\eta)\right), 
    \quad \xi\in L^\infty.
  \end{equation}
\end{corollary}
\begin{proof}
  The first assertion is clear from (\ref{eq:EstConj}), by which the
  conjugate $\mathcal{H}_{f^*,\gamma}$ is lower semicontinuous for any
  topology consistent with the duality $\langle L^\infty,L^1\rangle$,
  while (\ref{eq:DualRep1}) is a consequence of
  $\sigma(L^\infty,L^1)$-lower semicontinuity of
  $\mathcal{I}_{f,\gamma}$ (Lemma~\ref{lem:WellDefFatou}) via the
  Fenchel-Moreau theorem.
\end{proof}

In the classical case of Example~\ref{ex:Classical},
$\mathcal{D}_{f,\gamma} =\mathrm{dom}(\mathcal{I}_{f,\gamma})=\{\xi\in
L^\infty:f(\cdot,\xi)^+\in L^1\}$ (since $M^{\rho_\gamma}_u
=L^{\rho_\gamma}=L^1$), hence (\ref{eq:EstConj}) reduces to a single
equality which is exactly (\ref{eq:RockafellarClassical1}) as in the
Rockafellar theorem \citep[Theorem 1]{MR0310612}.  The original
version of \citep{MR0310612} is slightly more general, where the
integral functional is defined with respect to a
\emph{$\sigma$-finite} (rather than probability) measure $\mu$ for
$\mathbb{R}^d$-valued random variables $\xi\in
L^\infty(\Omega,\mathcal{F},\mu;\mathbb{R}^d)$. There are also some
extensions replacing $L^\infty(\Omega,\mathcal{F},\mu;\mathbb{R}^d)$
by some decomposable spaces of measurable functions taking values in a
Banach space. See in this line \citep{MR577677}, \citep{MR0467310},
and \citep{MR0512209} for a general reference.

In the polyhedral case of Example~\ref{ex:Polyhedral}
($\gamma=\delta_\mathcal{P}$ and $\mathcal{P}
=\mathrm{conv}(P_1,...,P_n)$), we still have $\mathcal{D}_{f,\gamma}
=\mathrm{dom}(\mathcal{I}_{f,\gamma}) =\textstyle\bigcap_{k\leq n}
\mathrm{dom}(I_{f,P_k})$. Thus (\ref{eq:EstConj}) reduces to
\begin{align}
  \label{eq:RobRockPoly1}
  \mathcal{I}^*_{f,\gamma}(\nu) =
  \min_{Q\in\mathcal{P}}\mathcal{H}_{f^*}(\nu_r|Q) +\sup\left\{
    \nu_s(\xi):\, \xi\in \textstyle\bigcap_{k\leq n}
    \mathrm{dom}(I_{f,P_k})\right\}.
\end{align}
This is slightly sharper than (\ref{eq:MaxFinitelyMany}) in the
sense that regular and singular parts are separated.

To the best of our knowledge, Rockafellar-type result for the robust
form (\ref{eq:SupIntFunc1}) of convex integral functionals (including
the homogeneous case of Example~\ref{ex:Coherent}) is new. A possible
complaint would be the difference between singular parts in the upper
and lower bounds in (\ref{eq:EstConj}). In the full generality of
Theorem~\ref{thm:RobustRockafellar}, however, \emph{both} inequalities
can really be strict and one can not hope for sharper bounds as the
next example illustrates (see Appendix~\ref{sec:DetailCounterEx} for
details).

\begin{example}[Badly Behaving Integrand]
  \label{ex:CounterEx1}
  Let $(\Omega,\mathcal{F}) :=(\mathbb{N},2^{\mathbb{N}})$ with
  $\mathbb{P}$ given by $\mathbb{P}(\{n\})=2^{-n}$, and $(P_n)_n$ a
  sequence of probability measures on $2^\mathbb{N}$ specified by
  $P_1(\{1\})=1$; $P_n(\{1\})=1-1/n$, $P_n(\{n\})=1/n$. Then
  $\mathcal{P} =\overline{\mathrm{conv}}(P_n;n\in\mathbb{N})$ is
  weakly compact in $L^1(\mathbb{N},2^{\mathbb{N}},\mathbb{P})$, thus
  $\gamma=\delta_\mathcal{P}$ is a penalty function satisfying
  Assumption~\ref{as:PenaltyFunc} and $\rho_\gamma(\xi)
  =\sup_n\mathbb{E}_{P_n}[\xi]$ if $\xi\in L^0_+$.  In this case,
  $L^\infty$ is regarded as the sequence space $\ell^\infty$ with the
  norm $\|\xi\|_\infty =\sup_n|\xi(n)|$, and $\nu\in
  ba^s_+(\mathbb{N},2^\mathbb{N},\mathbb{P})$ if and only if $\nu$
  vanishes on any finite set, or equivalently, for any $\nu\in
  ba_+$,
  \begin{equation}\label{eq:baSmallLinftySingular}
    \nu\in ba^s_+\,\Leftrightarrow\, \|\nu\|\cdot\liminf_n \xi(n)\leq \nu(\xi)\leq \|\nu\|\cdot \limsup_n\xi(n), 
    \,\,\forall \xi\in \ell^\infty =L^\infty.
  \end{equation}
  (Such $\nu\neq 0$ exists, thus $ba^s_+\setminus\{0\}\neq \emptyset$;
  see \cite[Lemmas~16.29 and 16.30]{aliprantis_border06}). Now we set
  \begin{equation}
    \label{eq:Ex4F}
    f(n,x)=nx^+e^x,\quad  n\in \mathbb{N}=\Omega,\, x\in \mathbb{R}.
  \end{equation}
  Then $\mathcal{I}_{f,\gamma}(\xi) =\sup_n\left(
    \left(1-\frac1n\right)\xi(1)^+e^{\xi(1)}
    +\xi(n)^+e^{\xi(n)^+}\right)\leq 2\|\xi\|_\infty
  e^{\|\xi\|_\infty}$, so $\mathrm{dom}
  (\mathcal{I}_{f,\gamma})=L^\infty$, and
  $\lim_{N\rightarrow\infty}\sup_n
  \mathbb{E}_{P_n}[f(\cdot,\xi)\ind_{\{f(\cdot,\xi)\geq N\}}]
  =\limsup_n \xi(n)^+e^{\xi(n)^+}$ (Lemma~\ref{lem:Ex4LimSup}), thus
  \begin{equation}
    \label{eq:CounterExDfgamma}
    0\in\mathcal{D}_{f,\gamma}
    =
    \{\xi\in\ell^\infty:\,\limsup_n\xi(n)\leq 0\} 
    \subsetneq \mathrm{dom}(\mathcal{I}_{f,\gamma}) =L^\infty.
  \end{equation}
  As for $\mathcal{I}_{f,\gamma}^*$,
  $\mathrm{dom}(\mathcal{I}_{f,\gamma}^*) \subset ba_+$ since
  $\mathcal{I}_{f,\gamma}$ is increasing, and
  $\mathcal{H}_{f^*,\gamma}(0) =0$ since
  $f^*(\cdot,0)=\inf_xf(\cdot,x)=0$, thus (\ref{eq:EstConj}) reads as
  $\sup_{\xi\in\mathcal{D}_{f,\gamma}} \nu(\xi)\leq
  \mathcal{I}_{f,\gamma}^*(\nu)\leq \sup_{\xi\in\mathrm{dom}
    (\mathcal{I}_{f,\gamma})}\nu(\xi)$ on $ba_+^s$. On the other hand,
  for $\nu\in ba^s_+$, $\sup_{\xi\in \mathcal{D}_{f,\gamma}}
  \nu(\xi)=0$, $\sup_{\xi\in \mathrm{dom}\mathcal{I}_f}\nu(\xi)
  =+\infty$, and (Lemma~\ref{lem:CountExConj1}):
  \begin{align*}
    \mathcal{I}_{f,\gamma}^*(\nu) &= \sup_{x\geq 0}
    x(\|\nu_s\|-e^x),\, \forall \nu\in ba^s_+.
  \end{align*}
  In particular, $\mathcal{I}_{f,\gamma}^*(\nu)=0$ if $\nu\in U^s_+
  :=\{\nu \in ba^s_+:\, \|\nu\|=\nu(\mathbb{N})\leq 1\}$,
  $0<\mathcal{I}_{f,\gamma}^*(\nu)<\infty$ if $\nu \in ba^s_+
  \setminus U^s_+ :=\{\nu\in ba^s_+:\, \|\nu\|>1\}$, and
  $\lim_{\|\nu\|\rightarrow\infty, \nu\in ba^s_+\setminus U^s_+}
  \mathcal{I}_{f,\gamma}^*(\nu)=\infty$. In summary,
  \begin{itemize}
  \item $\mathcal{I}_{f,\gamma}^*$ coincides with the lower bound
    $\sup_{\xi\in \mathcal{D}_{f,\gamma}}\nu(\xi)=0$ on $U^s_+$, while
  \item on $ba^s_+\setminus U^s_+$, $\mathcal{I}_{f,\gamma}^*$ is strictly
    between the upper and lower bounds and it runs through the whole
    interval of these bounds (in this specific case, $[0,\infty]$).
  \end{itemize}
\end{example}

\subsection{Finer Properties in the Finite-Valued Case}
\label{sec:FinValCase}

We now consider the regularities of $\mathcal{I}_{f,\gamma}$ and
$\mathcal{H}_{f^*,\gamma}$ in terms of the dual paring $\langle
L^\infty, L^1\rangle$.  In the classical case of
Example~\ref{ex:Classical}, the singular part of $I_f^*$ in
(\ref{eq:RockafellarClassical1}) is trivialized (i.e.,
$\delta_{\{0\}}$) as soon as $I_f :=\mathcal{I}_{f,\{\mathbb{P}\}}$ is
finite-valued, then $I_f^*$ reduces entirely to $I_{f^*}$. It implies
that all the sublevels of $I_{f^*}$ are $\sigma(L^1,L^\infty)$-compact
(see \citep[][Th.~3K]{MR0512209}), which is equivalent to the
continuity of $I_f$ for the Mackey topology $\tau(L^\infty,L^1)$, and
$I_f$ admits a $\sigma$-additive subgradient at every point (weak*
subdifferentiable).  Consequently, we can work entirely with the dual
pair $\langle L^\infty,L^1\rangle$.

In the robust case, the ``triviality of singular part of
$\mathcal{I}_{f,\gamma}^*$'' should be understood as
\begin{equation}
  \label{eq:elim1}
  \forall \nu\in ba,\,
  \mathcal{I}_{f,\gamma}^*(\nu)<\infty \,\Rightarrow \, \nu\text{ is $\sigma$-additive},
\end{equation}
i.e. that \emph{$\mathcal{I}_{f,\gamma}^*$ eliminates the singular
  measures}, which still makes sense even though
$\mathcal{I}_{f,\gamma}^*$ itself need not be the direct sum of
regular and singular parts.  This guarantees in particular that
(\ref{eq:EstConj}) reduces to a single equality (of course).  In the
case of Example~\ref{ex:CounterEx1}, $\mathcal{D}_{f,\gamma}
\subsetneq \mathrm{dom}(\mathcal{I}_{f,\gamma}) =L^\infty$, and
$\mathcal{I}_{f,\gamma}^*(\nu_s) <\infty$ as long as $\nu_s\geq
0$. Thus $\mathrm{dom}(\mathcal{I}_{f,\gamma}) =L^\infty$ is not
enough for (\ref{eq:elim1}), while from (\ref{eq:EstConj}),
$\mathcal{D}_{f,\gamma} =L^\infty$ is clearly sufficient.  In fact,
given the finiteness (plus a technical assumption),
$\mathcal{D}_{f,\gamma} = L^\infty$ is also necessary for
(\ref{eq:elim1}), and equivalent to other basic $\langle
L^\infty,L^1\rangle$-regularities of $\mathcal{I}_{f,\gamma}$ and
$\mathcal{I}_{f,\gamma}^*$ which follow solely from the finiteness in
the classical case.

\begin{theorem}
  \label{thm:LebMaxCompact}
  In addition to the assumptions of
  Theorem~\ref{thm:RobustRockafellar}, suppose
  $\mathrm{dom}(\mathcal{I}_{f,\gamma}) =L^\infty$ and
  \begin{equation}
    \label{eq:AddAssNegPart}
    \exists \xi_0'\in L^\infty\text{ with } f(\cdot,\xi'_0)^-\in M^{\rho_\gamma}.
  \end{equation}
  Then the following are equivalent:
  \begin{enumerate}[label=(\roman*),leftmargin=*,widest=viii,itemindent=0pt,labelindent=0pt]
  \item $\mathcal{D}_{f,\gamma} =L^\infty$, i.e., $f(\cdot,\xi)^+\in
    M^{\rho_\gamma}_u$ for all $\xi\in L^\infty$;
  \item $\mathbb{R}\subset \mathcal{D}_{f,\gamma}$, i.e.,
    $f(\cdot,x)^+\in M^{\rho_\gamma}_u$ for all $x\in\mathbb{R}$;
  \item $\mathcal{I}_{f,\gamma}^*$ eliminates singular measures in the
    sense of (\ref{eq:elim1});

  \item $\{\eta\in L^1:\, \mathcal{H}_{f^*,\gamma}(\eta)\leq c\}$ is
    $\sigma(L^1,L^\infty)$-compact for all $c\in\mathbb{R}$;
  \item $\mathcal{I}_{f,\gamma}$ is continuous for the Mackey topology
    $\tau(L^\infty, L^1)$;
  \item $\mathcal{I}_{f,\gamma}(\xi)
    =\lim_n\mathcal{I}_{f,\gamma}(\xi_n)$ if
    $\sup_n\|\xi_n\|_\infty<\infty$ and $\xi_n\rightarrow \xi$
    a.s. (the Lebesgue property);
  \item $\mathcal{I}_{f,\gamma}(\xi) =\max_{\eta\in L^1}
    \left(\mathbb{E}[\xi\eta]-\mathcal{H}_{f^*,\gamma}(\eta)\right)$,
    i.e., the supremum is attained in (\ref{eq:DualRep1}),
    $\forall\xi\in L^\infty$.
  \end{enumerate}
  Here implications (i) $\Leftrightarrow$ (ii) $\Rightarrow$ (iii)
  $\Rightarrow$ (iv) $\Leftrightarrow$ (v) $\Rightarrow$ (vi) and (v)
  $\Rightarrow$ (vii) are true without (\ref{eq:AddAssNegPart}).
\end{theorem}
A proof will be given in Section~\ref{sec:ProofLebesgue}.
\begin{remark}
  \label{rem:JamesML}
  The finiteness of $\mathcal{I}_{f,\gamma}$ already implies that
  $f(\cdot,\xi)^+\in L^{\rho_\gamma}$, $\forall \xi\in L^\infty$ (in
  particular $f$ is finite-valued), while the additional assumption
  (\ref{eq:AddAssNegPart}) is made to guarantee $f(\cdot, \xi)^+\in
  M^{\rho_\gamma}$ for all $\xi\in L^\infty$ (see
  (\ref{eq:ConseqAddAss1}) and the subsequent paragraph). These
  coincide in the homogeneous case (Example~\ref{ex:Coherent};
  including the classical case) since then
  $L^{\rho_\gamma}=M^{\rho_\gamma}$, but in general, $M^{\rho_\gamma}
  \subsetneq L^{\rho_\gamma}$ is possible. A sufficient condition for
  (\ref{eq:AddAssNegPart}) is that
  \begin{equation}
    \label{eq:EtaMrho1}
    \exists \eta_0\in M^{\rho_\gamma}\text{ with } f^*(\cdot, \eta_0)^+\in M^{\rho_\gamma},
  \end{equation}
  (then $f(\cdot, \xi)^-\in M^{\rho_\gamma}$ for \emph{all} $\xi\in
  L^\infty$), which is identical to (\ref{eq:AsConjInteg1}) (contained
  in the standing assumptions) in the homogeneous case.
\end{remark}

\begin{remark}
  The equivalence between (iv), (vi) and (vii) for \emph{convex risk
    measures} on $L^\infty$, i.e., for $\rho_\gamma|_{L^\infty}$ with
  penalty function $\gamma$ as in Assumption~\ref{as:PenaltyFunc} is
  known (\citep{jouini06:_law_fatou} and \citep{MR2648597}) followed
  by some generalizations:
  \citep{MR2847443,MR2924150} for
  convex risk measures on Orlicz spaces, and
  \citep{MR3165235,MR3177424} for \emph{finite-valued} monotone convex
  functions on solid spaces of measurable functions among others.
\end{remark}

From (ii) $\Rightarrow$ (v), we derive a simple criterion in terms of
integrability of $f$ for the $\langle L^\infty,L^1\rangle$-Fenchel
duality for the minimization of robust integral functional
$\mathcal{I}_{f,\gamma}$. Here we recall Fenchel's duality theorem
(see \citep[Th.~1]{MR0187062}): if $\langle E,E'\rangle$ is a dual
pair, $\varphi, \psi$ are proper convex functions on $E$, and if
either $\varphi$ or $\psi$ is $\tau(E,E')$-continuous at some $x\in
\mathrm{dom}(\varphi)\cap \mathrm{dom}(\psi)$, then
\begin{align*}
  \inf_{x\in E} \left(\varphi(x)+\psi(x)\right) =-\min_{x'\in E'}
  \left(\varphi^*(x')+\psi^*(-x')\right).
  \end{align*}
  Putting $E=L^\infty$, $E'=L^1$, $\varphi =\mathcal{I}_{f,\gamma}$
  and $\psi=\delta_\mathcal{C}$ with $\mathcal{C}\subset L^\infty$
  convex, (ii) $\Rightarrow$ (v) tells us that

\begin{corollary}[Fenchel Duality]
  \label{cor:FenchelDuality}
  Let $\gamma$ be a penalty function satisfying
  Assumption \ref{as:PenaltyFunc} and $f$ a proper normal convex
  integrand. If $f(\cdot, x)^+\in M^{\rho_\gamma}_u$ for all
  $x\in\mathbb{R}$, and $f^*(\cdot, \eta_0)^+\in L^{\rho_\gamma}$ for
  some $\eta_0\in L^{\rho_\gamma}$, then for any convex set
  $\mathcal{C}\subset L^\infty$,
  \begin{equation*}
    \inf_{\xi\in \mathcal{C}}\mathcal{I}_{f,\gamma}(\xi) 
    =-\min_{\eta\in L^1} \Bigl(\mathcal{H}_{f^*,\gamma}(-\eta) 
    +\sup_{\xi'\in \mathcal{C}}\mathbb{E}[\xi'\eta]\Bigr).
  \end{equation*}
  If in addition $\mathcal{C}$ is a convex cone, the right hand side
  is equal to $-\min_{\eta\in \mathcal{C}^\circ}
  \mathcal{H}_{f^*,\gamma}(-\eta)$, where $\mathcal{C}^\circ
  =\{\eta\in L^1:\, \mathbb{E}[\xi\eta]\leq 1,\,\forall \xi\in
  \mathcal{C}\}$ (the one-sided polar of $\mathcal{C}$ in $\langle
  L^\infty, L^1\rangle$).
\end{corollary}

The \emph{subdifferential of $\mathcal{I}_{f,\gamma}$ at $\xi\in
  L^\infty$} is the following set of $\nu\in (L^\infty)^*$ called
subgradients:
\begin{equation*}
  \partial\mathcal{I}_{f,\gamma}(\xi) :=\{\nu\in (L^\infty)^*:\,
  \nu(\xi)-\mathcal{I}_{f,\gamma}(\xi) \geq \nu(\xi') 
  -\mathcal{I}_{f,\gamma}(\xi'),  \, \forall \xi'\in L^\infty\}.
\end{equation*}
We say that $\mathcal{I}_{f,\gamma}$ is subdifferentiable at $\xi$ if
$\partial\mathcal{I}_{f,\gamma}(\xi) \neq\emptyset$. In view of
(\ref{eq:DualRep1}), $\eta\in \partial\mathcal{I}_{f,\gamma}(\xi)\cap
L^1$ (then $\eta$ is called a \emph{$\sigma$-additive subgradient of
  $\mathcal{I}_{f,\gamma}$ at $\xi$}) if and only if it maximizes
$\eta'\mapsto \mathbb{E}[\xi\eta'] -\mathcal{H}_{f^*,\gamma}(\eta')$,
thus (vii) is equivalent to saying that for every $\xi\in L^\infty$,
$\partial \mathcal{I}_{f,\gamma}(\xi) \cap L^1\neq \emptyset$. Note
also that $\partial\mathcal{I}_{f,\gamma}(\xi) \subset
\mathrm{dom}(\mathcal{I}_{f,\gamma}^*)$ since
$\mathcal{I}^*_{f,\gamma}(\nu) =\sup_{\xi'\in L^\infty}(\nu(\xi')
-\mathcal{I}_{f,\gamma}(\xi'))\leq \nu(\xi)
-\mathcal{I}_{f,\gamma}(\xi)<\infty$ if
$\nu\in \partial\mathcal{I}_{f,\gamma}(\xi)$. Thus (iii) implies
$\partial\mathcal{I}_{f,\gamma}(\xi)\subset L^1$. Summing up,
\begin{corollary}
  \label{cor:Subdiff}
  Under the assumptions of Theorem~\ref{thm:LebMaxCompact}, (i) -- (vii)
  are equivalent also to
  \begin{enumerate}[label=(\roman*),leftmargin=*,widest=viii,itemindent=0pt,labelindent=0pt]
    \setcounter{enumi}{7}
  \item $\emptyset\neq \partial \mathcal{I}_{f,\gamma}(\xi)\subset
    L^1$ for every $\xi\in L^\infty$.
  \end{enumerate}

\end{corollary}

The weak compactness of the sublevels of $\mathcal{H}_{f^*,\gamma}$
can be viewed as a generalization of the \emph{de la Vallée-Poussin
  theorem} which asserts that a set $\mathcal{C}\subset L^1$ is
uniformly integrable if and only if there exists a function
$g:\mathbb{R}\rightarrow(-\infty,\infty]$ which is \emph{coercive}:
$\lim_{|y|\rightarrow \infty}g(y)/y=\infty$ and $\sup_{\eta\in
  \mathcal{C}}\mathbb{E}[g(|\eta|)]
=\sup_{\eta\in\mathcal{C}}\mathcal{H}_{g,\delta_{\{\mathbb{P}\}}}(|\eta|)
<\infty$ (e.g. \citep[Th.~II.22]{dellacherie_meyer78}).  The
\emph{coercivity condition} is equivalent to saying that
$\mathrm{dom}(g^*) =\mathbb{R}$. Now we have as a consequence of (ii)
$\Leftrightarrow$ (iv):

\begin{corollary}[cf. \citep{MR2247836} when $g$ is non-random]
  \label{cor:delaVallee}
  A set $\mathcal{C}\subset L^1$ is uniformly integrable if and only
  if there exists a convex penalty function $\gamma$ on $\mathcal{Q}$
  satisfying Assumption~\ref{as:PenaltyFunc} as well as a proper
  normal convex integrand $g$ with
  $g^*(\cdot,x)\in M^{\rho_\gamma}_u$, $\forall x\in\mathbb{R}$, such
  that
  $\sup_{\eta\in\mathcal{C}}\mathcal{H}_{g,\gamma}(\eta) <\infty$.
\end{corollary}
\begin{proof}
  Let $f=g^*$, then $f^*=g^{**}=g$ (since normal). Then
  $\mathcal{D}_{f,\gamma} =L^\infty$ by assumption and
  $\mathcal{H}_{g,\gamma} =\mathcal{H}_{f^*,\gamma}$ is well-defined
  while $\sup_{\eta\in\mathcal{C}}\mathcal{H}_{g,\gamma}(\eta)<\infty$
  guarantees (\ref{eq:AsConjInteg1}) as well. Now the sufficiency is
  nothing but (ii) $\Rightarrow$ (iv), while the necessity is clear from
  the above paragraph.
\end{proof}

\subsection{Examples of ``Nice'' Integrands and Robust Utility
  Maximization}
\label{sec:ExNiceInt}

When $f$ is \emph{non-random} and finite, $\mathbb{R}\subset
\mathcal{D}_{f,\gamma}$ is automatic, while $f^*(y)\in
M^{\rho_\gamma}_u$ for any $y\in\mathrm{dom}f^*\neq \emptyset$ (since
constant).  Here are some ways to generate ``nice'' random integrands.

\begin{example}[Random scaling]
  \label{ex:IntegrandDiscount}
  Let $g:\mathbb{R}\rightarrow\mathbb{R}$ be a (non-random) finite
  convex function $\not\equiv0$, and $W\in L^0$ be strictly positive
  (i.e., $\mathbb{P}(W>0)=1$). Then put
  \begin{equation*}
    f(\omega,x) :=g(W(\omega)x), \quad \forall (\omega,x)\in \Omega\times\mathbb{R}.
  \end{equation*}
  In this case, $f^*(\omega,y) =g^*(y/W(\omega))$ and
  $\mathbb{R}\subset \mathcal{D}_{f,\gamma}$ is true if
  \begin{equation}
    \label{eq:ScalingCond1}
    \exists \delta>0,\, p>1\text{ such that } g(-\delta W^p)^+\vee g(\delta W^p)^+\in M^{\rho_\gamma}_u.
  \end{equation}
  Note that
  $|Wx|=\frac\delta2|W(2 x/\delta)| \leq \frac12\left(\delta
    W^p+\frac{2^{q}}{\delta^{q-1}}|x|^q\right)$
  where $\frac1p+\frac1q=1$. Applying the (quasi) convexity of $g$
  twice, (\ref{eq:ScalingCond1}) implies for each $x\in\mathbb{R}$
  \begin{align*}
    g(Wx)\leq g(-\delta W^p)^+ \vee g(\delta W^p)^+\vee
    g\left(-\frac{2^{q}}{\delta^{q-1}}|x|\right)^+ \vee
    g\left(\frac{2^{q}}{\delta^{q-1}}|x|\right)^+ \in
    M^{\rho_\gamma}_u.
  \end{align*}
  Also, since $g\not\equiv0$, $\mathrm{dom} g^*\setminus
  \{0\}\neq\emptyset$. If $y\in \mathrm{dom} g^*$ and $y>0$
  (resp. $y<0$),
  \begin{align*}
    0\leq W\leq 1+W^p\leq 1+ \frac{g(\delta W^p)^+
      +g^*(y)}{y\delta};\quad\text{resp. } \leq 1+ \frac{g(-\delta
      W^p)^+ +g^*(y)}{-y\delta}.
  \end{align*}
  In both cases, (\ref{eq:ScalingCond1}) implies $W\in
  M^{\rho_\gamma}$, and consequently, $\eta_y =yW\in M^{\rho_\gamma}$
  and $f^*(\cdot, \eta_y)=g^*(y)\in L^\infty$. Thus
  (\ref{eq:EtaMrho1}) ($\Rightarrow$ (\ref{eq:AsConjInteg1})) follows
  from (\ref{eq:ScalingCond1}) as well.  If in addition $g$ is
  monotone increasing, $g(-W^p)^+\leq g(0)^+$, thus the half of
  (\ref{eq:ScalingCond1}) is automatically true.
\end{example}

\begin{example}[Random parallel shift]
  \label{ex:IntegTrans1}
  Let $f$ be a finite normal convex integrand satisfying
  (\ref{eq:AsConjInteg1}) and $\mathbb{R}\subset \mathcal{D}_f$, and
  $B\in L^0$. Then put
  \begin{equation}
    \label{eq:IngegFB}
    f_B(\omega,x)=f(\omega,x+B(\omega)),
    \quad (\omega,x)\in \Omega\times\mathbb{R}.
  \end{equation}
  By convexity of $f$, $f(\cdot,x+B)\leq
  \frac{\varepsilon}{1+\varepsilon} f\left(\cdot,
    \frac{1+\varepsilon}\varepsilon x\right) +\frac1{1+\varepsilon}
  f(\cdot, (1+\varepsilon)B)$ and $f\left(\cdot,
    \frac\varepsilon{1+\varepsilon}x\right)\leq
  \frac\varepsilon{1+\varepsilon} f(\cdot, x+B) +\frac1{1+\varepsilon}
  f(\cdot, -\varepsilon B)$, thus putting $\Gamma_\alpha(x) =f(\cdot,
  \alpha x)^+/\alpha$,
  \begin{align}
    \label{eq:FBEstim1}
    \frac{1+\varepsilon}{\varepsilon} f\left(\cdot,
      \frac{\varepsilon}{1+\varepsilon}x\right) -\Gamma_\varepsilon
    (-B)&\leq f_B(\cdot,x) \leq \frac\varepsilon{1+\varepsilon}
    f\left(\cdot, \frac{1+\varepsilon}\varepsilon x\right)
    +\Gamma_{1+\varepsilon} (B),\\
       \label{eq:FBEstim2}
       f^*(\cdot, y)-\Gamma_{1+\varepsilon}(B)&\leq f_B^*(\cdot, y)
       \leq f^*(\cdot, y)+\Gamma_{\varepsilon}(-B),
  \end{align}
  where $f^*_B(\cdot,y) =f^*(\cdot,y)-yB$, and (\ref{eq:FBEstim2})
  follows from (\ref{eq:FBEstim1}) by taking conjugates.  Thus if
  \begin{align}
    \label{eq:IntegCondFB1}
    &\exists \varepsilon>0\text{ such that } \Gamma_{1+\varepsilon}(B)
    \in M^{\rho_\gamma}_u \text{ and } \Gamma_\varepsilon(-B)\in
    L^{\rho_\gamma},
  \end{align}
  then $\mathbb{R}\subset \mathcal{D}_{f_B,\gamma}$ and $f_B^*$
  satisfies (\ref{eq:AsConjInteg1}) ((\ref{eq:EtaMrho1}) if $f^*$
  does). Moreover, (\ref{eq:FBEstim2}) implies in this case
  \begin{equation*}
    \mathcal{H}_{f^*_B,\gamma}(\eta)<\infty
    \,\Leftrightarrow\, \mathcal{H}_{f^*,\gamma}(\eta)<\infty\,
    \Rightarrow\, \eta B\in L^1,
  \end{equation*}
  and $\mathcal{H}_{f^*_B,\gamma}$ is explicitly given in terms of
  $\mathcal{H}_{f^*,\gamma}$ as
  \begin{equation}
    \label{eq:FBConj}
    \mathcal{H}_{f^*_B,\gamma}(\eta)=
    \mathcal{H}_{f^*,\gamma}(\eta)-\mathbb{E}[\eta B],
    \quad \forall \eta\in \mathrm{dom}(\mathcal{H}_{f^*_B,\gamma}) 
    =\mathrm{dom}(\mathcal{H}_{f^*,\gamma}).
  \end{equation}
\end{example}

We can combine the preceding two examples:
\begin{example}
  \label{ex:FBW}
  Let $g:\mathbb{R}\rightarrow\mathbb{R}$, $W$ and $B$ be as in
  Examples~\ref{ex:IntegrandDiscount}~and~\ref{ex:IntegTrans1}, and
  put
  \begin{equation*}
    h(\cdot,x):=g(Wx+B)=g(W(x+B/W))
  \end{equation*}
  This $h$ satisfies (\ref{eq:AsConjInteg1}) and
  $\mathbb{R}\subset\mathcal{D}_{h,\gamma}$ if $(g,W)$ satisfies
  (\ref{eq:ScalingCond1}) and (\ref{eq:IntegCondFB1}) holds with
  $f=g$.  Note that if we apply Example~\ref{ex:IntegTrans1} to
  $f(\cdot,x)=g(Wx)$ and $B/W$, then $h(\cdot, x) =f(\cdot,x+ B/W)$
  and e.g.  $f(\cdot, (1+\varepsilon)B/W) =g((1+\varepsilon)B)$.
\end{example}

Our initial motivation was a duality method for \emph{robust utility
  maximization} of the general form (\ref{eq:RobUtilIntro1}) with
\emph{random utility function} $U:\Omega\times\mathbb{R}
\rightarrow\overline{\mathbb{R}}$.  See \cite{follmer_schied_weber09}
for the financial background of the problem.  A motivational example
of random utility is of the type $U_{D,B}(\cdot,x)=U(D^{-1}x+B)$ where
$U:\mathbb{R}\rightarrow\overline{\mathbb{R}}$ is a proper concave
increasing function and $D,B\in L^0$ (with $D>0$ a.s.)  correspond
respectively to the discount factor and a payoff of a claim. Then the
problem is to
\begin{equation}
  \label{eq:RobUtilAbst1}
  \text{maximize}\quad 
  u_{D,B,\gamma}(\xi) 
  :=\inf_{Q\in\mathcal{Q}} 
  \left(\mathbb{E}_Q[U(D^{-1}\xi+B)]+\gamma(Q)\right) 
  =-\mathcal{I}_{f_{D,B},\gamma}(-\xi)
\end{equation}
over a convex cone $\mathcal{C}\subset L^\infty$ where $f_{D,B}(\cdot,
x)=-U(-D^{-1}x+B)$ is a proper normal convex integrand of the form in
Example~\ref{ex:FBW}. The full detail of this problem in more concrete
financial setup will be given in a separate future paper together with
an application to a robust version of \emph{utility indifference
  valuation}.  Here we just give a criterion in terms of
``integrabilities'' of $D$ and $B$ for the duality without singular
term as well as its explicit form \emph{when $U$ is finite on
  $\mathbb{R}$}. It constitutes a half of what we call the martingale
duality method (see e.g.  \cite{MR1883783,MR2014244,MR2830428} for the
other half in the classical case and
\citep{MR2966354}\footnote{\unskip There an earlier version of this
  paper (still available as \href{http://arxiv.org/abs/arXiv:1101.2968}{arXiv:1101.2968}) was used.} for a partial
result in the robust case). For the case $\mathrm{dom}(U)
=\mathbb{R}_+$, see \citep{MR2284014} when $D,B$ are constants;
\citep{wittmuss08} with bounded $B$, and \citep{MR2247836} for
$\mathrm{dom}(U) =\mathbb{R}$ with constant $D,B$; see also
\citep{follmer_schied_weber09} for more thorough references. The
following is an immediate consequence of
Corollary~\ref{cor:FenchelDuality} and Example~\ref{ex:FBW}.

\begin{corollary}
  \label{cor:DualityUtil1}
  In the above notation, suppose $U$ is finite on $\mathbb{R}$, and
  \begin{equation}
    \label{eq:RobUtilSuffCond}
    \exists \delta,\varepsilon>0\text{ with }  
    U\bigl(-\delta D^{-(1+\varepsilon)}\bigr)^- 
    \in M^{\rho_\gamma}_u,\, 
    U(-(1+\varepsilon)B)^-\in M^{\rho_\gamma}_u,\,  
    U(\varepsilon B)^-\in L^{\rho_\gamma}.
  \end{equation}
  Then for any nonempty convex cone $\mathcal{C}\subset L^\infty$, it
  holds that
  \begin{align}
    \label{eq:RobUtilDuality1}
    \sup_{\xi\in \mathcal{C}}u_{B,D,\gamma}(\xi) = \min_{\eta\in
      \mathcal{C}_V^\circ} \left(\mathcal{H}_{V,\gamma}(D\eta)
      +\mathbb{E}[D B\eta]\right),
  \end{align}
  where $V(y):=\sup_{x\in \mathbb{R}}(U(x)-xy)$ and
  $\mathcal{C}^\circ_V
  :=\{\eta\in\mathcal{C}^\circ:\,\mathcal{H}_{V,\gamma}(D\eta)<\infty\}$.
\end{corollary}

Here $\mathcal{C}$ can be any convex cone and the possibility of both
sides being $\infty$ is not excluded; it does not happen iff
$\mathcal{C}^\circ_V\neq \emptyset$.  If in addition
$\mathcal{C}^{\circ,e}_V :=\{\eta\in\mathcal{C}^\circ_V:\,\eta>0
\text{ a.s.}\} \neq \emptyset$, we can replace ``$\min_{\eta\in
  \mathcal{C}^\circ_V}$'' by
``$\inf_{\eta\in\mathcal{C}^{\circ,e}_V}$'' etc with a little more
effort and certain regularities of $U$. Choosing a ``good'' cone
$\mathcal{C}$, these conditions as well as the dual problem have clear
financial interpretations and consequences (see \citep{MR1883783} for
a good exposition in the classical case).

A couple of features deserve attention: (i) We directly invoke
Fenchel's theorem to the functional $u_{D,B,\gamma}
=-\mathcal{I}_{f_{D,B},\gamma}(-\cdot)$ by means of our main theorem,
instead of interchanging ``$\sup_{\xi\in\mathcal{C}}$'' and
``$\inf_{Q\in\mathcal{Q}}$'', and invoke a \emph{classical} duality
(see \citep{MR1883783,MR2830428} and its references) under each $Q$,
so we do not need to mind what happens under ``extreme''
$Q\in\mathcal{Q}_\gamma$; (ii) embedding the randomness $D,B$ to the
utility function $U$ instead of transforming the domain $\mathcal{C}$
to $D^{-1}\mathcal{C}+B$, we retain the ``good form'' of $\mathcal{C}$
(which is essential for the probabilistic techniques to work well),
and obtain a criterion for the duality in terms solely of $B$ and $D$
(when $U$ is finite). Those integrability conditions are weak even in
the classical case where $\gamma =\delta_{\{\mathbb{P}\}}$ and
$D\equiv 1$ (then (\ref{eq:RobUtilSuffCond}) reads as
$U(-(1+\varepsilon)B)^-, U(\varepsilon B)^-\in L^1$ for \emph{some}
$\varepsilon>0$ complementing the result of \citep{MR2830428}).

\section{Proofs}
\label{sec:Proofs}

\subsection{Proof of Theorem~\ref{thm:RobustRockafellar}}
\label{sec:ProofIntRepMain}

The upper bound is simply a consequence of Young's inequality
(\ref{eq:YoungIneqRob}) and $\nu=\nu_r+\nu_s$:
\begin{align*}
  \mathcal{I}_{f,\gamma}^*(\nu) &
  =\sup_{\xi\in\mathrm{dom}(\mathcal{I}_{f,\gamma})}
  \left(\nu_r(\xi)-\mathcal{I}_{f,\gamma}(\xi)+\nu_s(\xi)\right)
  \stackrel{\text{(\ref{eq:YoungIneqRob})}}{\leq}
  \mathcal{H}_{f^*,\gamma}(\nu_r)
  +\sup_{\xi\in\mathrm{dom}(\mathcal{I}_{f,\gamma})}\nu_s(\xi).
\end{align*}

The lower bound is more involved. First, fix $\nu\in ba$ and define
\begin{align*}
  L_\nu(Q,\xi) :=\nu(\xi)- \mathbb{E}_Q[f(\cdot, \xi)]+\gamma(Q),\quad
  Q\in \mathcal{Q}_\gamma,\, \xi\in\mathcal{D}_{f,\gamma}.
\end{align*}
It is finite-valued on
$\mathcal{Q}_\gamma\times\mathcal{D}_{f,\gamma}$, convex in
$Q\in\mathcal{Q}_\gamma$ and concave in $\xi\in
\mathcal{D}_{f,\gamma}$.  Moreover,
\begin{lemma}
  \label{lem:ProofRockafellarMinimax1}
  With the notation above, it holds that
  \begin{align}
    \label{eq:ProofMainMinimax1}
    \inf_{Q\in\mathcal{Q}_\gamma}
    \sup_{\xi\in\mathcal{D}_{f,\gamma}}L_\nu(Q,\xi)
    =\sup_{\xi\in\mathcal{D}_{f,\gamma}}
    \inf_{Q\in\mathcal{Q}_\gamma}L_\nu(Q,\xi).
  \end{align}
\end{lemma}
\begin{proof}
  We claim that for any $\xi\in \mathcal{D}_{f,\gamma}$ and $c>0$,
  $\Lambda_c(\xi) :=\{Q\in\mathcal{Q}_\gamma:\, L_\nu(Q,\xi) \leq c\}$
  is weakly compact in $L^1$. Again let $\psi_Q :=dQ/d\mathbb{P}$ for
  each $Q\in\mathcal{Q}_\gamma$.  For the uniform integrability of
  $\Lambda_c(\xi)$, pick a $Q_0\in \mathcal{Q}_\gamma$ with
  $\gamma(Q_0)=0$ (see Remark~\ref{rem:PenaltyFuncAssump}) and observe
  that
  \begin{align}
    \label{eq:ProofMinimax1}
    \mathbb{E}_Q[f(\cdot, \xi)] \leq
    \mathbb{E}\left[2f(\cdot,\xi)^+\frac{\psi_Q+\psi_{Q_0}}2\right]
    \leq \rho_\gamma\left(2f(\cdot, \xi)^+\right)
    +\frac12\gamma(Q)<\infty,
  \end{align}
  since $\xi\in\mathcal{D}_{f,\gamma}$. Thus $\Lambda_c(\xi)
  \subset\left\{Q\in\mathcal{Q}_\gamma: \gamma(Q)\leq
    2\left(c-\nu(\xi) +\rho_\gamma(2 f(\cdot,\xi)^+)\right)\right\}$,
  hence $\Lambda_c(\xi)$ is uniformly integrable by
  (\ref{eq:PenaltyFuncCompact}) and so is $\Lambda'_c(\xi)
  :=\left\{f(\cdot, \xi)^+\psi_Q:\, L_\nu(Q,\xi)\leq c\right\}$ by
  (\ref{eq:UIOrliczIFF}) since $\xi\in\mathcal{D}_{f,\gamma}$.

  To see that $\Lambda_c(\xi)$ is weakly closed, it suffices to show
  that it is norm-closed since convex. So let $(Q_n)_n$ be a sequence
  in $\Lambda_c$ converging in $L^1$ to some $Q$, then passing to a
  subsequence, we can suppose without loss that $\psi_{Q_n}\rightarrow
  \psi_Q$ a.s. too. From the previous paragraph,
  $(f(\cdot,\xi)^+\psi_{Q_n})_n$ is uniformly integrable, hence by the
  (reverse) Fatou lemma, we have $\mathbb{E}_Q[f(\cdot, \xi)]\geq
  \limsup_n \mathbb{E}_{Q_n}[f(\cdot,\xi)]$ and consequently,
  \begin{align*}
    L_\nu(Q,\xi) &=\nu(\xi)-\mathbb{E}_Q[f(\cdot, \xi)]+\gamma(Q)\\
    &\leq\nu(\xi)+\liminf_n -\mathbb{E}_{Q_n}[f(\cdot,\xi)]
    +\liminf_n\gamma(Q_n)\\
    &\leq \liminf_n\left(\nu(\xi)-\mathbb{E}_{Q_n}[f(\cdot, \xi)]
      +\gamma(Q_n)\right)\leq c.
  \end{align*}
  We deduce that $\Lambda_c(\xi)$ is weakly closed, hence weakly
  compact. Now (\ref{eq:ProofMainMinimax1}) follows from a minimax
  theorem (see \citep[Appendix~A]{MR3165235}).
\end{proof}

Noting that $\mathcal{I}_{f,\gamma}^*(\nu)\geq
\sup_{\xi\in\mathcal{D}_{f,\gamma}}
(\nu(\xi)-\mathcal{I}_{f,\gamma}(\xi))
=\sup_{\xi\in\mathcal{D}_{f,\gamma}} \inf_{Q\in\mathcal{Q}_\gamma}
L_\nu(Q,\xi)$, we deduce from Lemma~\ref{lem:ProofRockafellarMinimax1}
that for any $\nu\in ba$,
  \begin{equation}
    \label{eq:Minimax1}
    \mathcal{I}_{f,\gamma}^*(\nu) \geq
    \inf_{Q\in\mathcal{Q}_\gamma} \sup_{\xi\in\mathcal{D}_{f,\gamma}} L_\nu(Q,\xi)
    =\inf_{Q\in\mathcal{Q}_\gamma}\sup_{\xi\in\mathcal{D}_{f,\gamma}} \left(\nu(\xi) -\mathbb{E}_Q\left[f(\cdot,
        \xi)\right] +\gamma(Q)\right).
  \end{equation}

\begin{lemma}
  \label{lem:MbleSelection1}
  Let $\eta,\zeta,\psi\in L^1$ such that $\psi\geq 0$ a.s. and
  \begin{equation}
    \label{eq:InfiniteMbleSelec1}
    \tilde f^*\left(\cdot,\eta,\psi\right) 
    =\sup_{x\in\mathrm{dom}f}(x\eta-\psi f(\cdot,x))>\zeta\text{ a.s.}
  \end{equation}
  Then there exists a $\hat\xi\in L^0$ such that
  \begin{equation}
    \label{eq:InfiniteMbleSelec2}
    f(\cdot,\hat\xi)<\infty\text{ a.s. and } \hat\xi \eta-\psi(\cdot,\hat\xi)\geq \zeta\text{ a.s. }
  \end{equation}

\end{lemma}
\begin{proof}
  This amounts to proving that the multifunction
  \begin{align*}
    S(\omega):=\{x\in \mathrm{dom}f(\omega,\cdot): \,
    x\eta(\omega)-\psi(\omega)f(\omega,x)\geq \zeta(\omega)\}
  \end{align*}
  admits a measurable selection.  $S$ is nonempty valued by
  (\ref{eq:InfiniteMbleSelec1}), and measurable since
  $g(\omega,x):=\psi(\omega)f(\omega,x)-x\eta(\omega)$ (with the
  convention $0\cdot\infty=0$) is a normal convex integrand (see
  \citep[][Prop.~14.44, Cor.~14.46]{rockafellar_wets98}), and
  $S(\omega) =\mathrm{dom}f(\omega,\cdot)\cap \{x:\, g(\omega,x)\leq
  -\zeta(\omega)\}$. On $\{\psi>0\}$, we have simply $S=\{x:\,
  f(\cdot,x)-x\frac{\eta}{\psi}\leq -\frac{\zeta}{\psi}\}$ which is
  \emph{closed} since $f$ is normal. Thus
  \begin{align*}
    S'(\omega)=
    \begin{cases}
      S(\omega)&\text{ if }\omega\in \{\psi>0\},\\
      \emptyset&\text{ if }\omega\in \{\psi=0\},
    \end{cases}
  \end{align*}
  is a closed-valued measurable multifunction with $\mathrm{dom}
  S'=\{\omega:\, S'(\omega)\neq \emptyset\}=\{\psi>0\}$, thus the
  standard measurable selection theorem (see
  \citep{rockafellar_wets98}, Cor.~14.6) shows the existence of
  $\xi'\in L^0$ such that $\xi'(\omega)\in S'(\omega)=S(\omega)$ for
  $\omega\in\{\psi>0\}$.

  On $\{\psi=0\}$, the multifunction $S$ need not be closed-valued. So
  we explicitly construct a selector. First, on $\{\psi=\eta=0\}$, we
  can take any $\xi_0\in
  \mathrm{dom}\mathcal{I}_{f,\gamma}$($\neq\emptyset$ by assumption),
  which satisfies $\xi(\omega)\in \mathrm{dom} f(\omega,\cdot)$ for
  a.e. $\omega$ by (\ref{eq:DomIFImplyDomF}), and $\xi_0\eta-\psi
  f(\cdot, \xi_0)=0>\zeta$ on $\{\psi=\eta=0\}$ since
  (\ref{eq:InfiniteMbleSelec1}) reads as $\zeta<0$ when $\psi=\eta=0$.

  Next, we put
  \begin{align*}
    \xi'':=\frac12\left(\left(a^+_f\wedge \frac\zeta\eta\right)+
      \left(a^-_f\vee \frac\zeta\eta\right)\right)\quad \text{on }
    \{\psi=0,\eta\neq 0\},
  \end{align*}
  where $a^\pm_{f^*}(\omega)=\lim_{x\rightarrow
    \pm\infty}f^*(\omega,x)/x$ as in (\ref{eq:ConjTilde3}). Recalling
  that $\mathrm{dom} f(\omega,\cdot) =\{a^+_{f^*}(\omega)\}$ if
  $a^-_{f^*}(\omega) =a^+_{f^*}(\omega)$, and otherwise
  $(a^-_{f^*}(\omega),a^+_{f^*}(\omega)) \subset \mathrm{dom}
  f(\omega,\cdot)\subset [a^-_{f^*}(\omega),a^+_{f^*}(\omega)]$, we
  have $\xi''(\omega)\in \mathrm{dom} f(\omega,\cdot)$ for $\omega\in
  \{\psi=0,\eta\neq 0\}$. Also, (\ref{eq:InfiniteMbleSelec1}) reads as
  $a^+_f\eta>\zeta$ when $\psi=0$ and $\eta>0$, hence
  \begin{align*}
    \xi''\eta-\psi f(\cdot,\xi'') = \xi''\eta &=\frac12
    \zeta+\frac12\left(a^-_f\vee \zeta\right)\geq\zeta\quad \text{on }
    \{\psi=0,\eta>0\}.
  \end{align*}
  Similarly, (\ref{eq:InfiniteMbleSelec1}) reads as $a^-_f\eta>\zeta$
  on $\{\psi=0,\eta<0\}$, hence
  \begin{align*}
    \xi''\eta&=\frac12\left(\left(a^+_f\eta\vee \zeta\right)
      +\left(a^-_f\eta\wedge \zeta\right)\right) =\frac12 \left(
      a^+_f\vee \zeta\right) +\frac12\zeta\geq \zeta\quad \text{on }
    \{\psi=0,\eta<0\}.
  \end{align*}
  Now $\hat\xi :=\xi'\ind_{\{\psi>0\}} + \xi_0\ind_{\{\psi=0,\eta=0\}}
  +\frac12\xi''\ind_{\{\psi=0, \eta\neq 0\}}$ is a desired measurable
  selection.
\end{proof}

\begin{proof}[Proof of Theorem~\ref{thm:RobustRockafellar}]
  The upper bound is already established at the beginning of this
  section. For the lower bound, it suffices to show that for any
  $\nu=\nu_r+\nu_s\in ba$,
  \begin{equation}
    \label{eq:ProofMainThKey1}
    \alpha<\mathcal{H}_{f^*,\gamma}(\nu_r),\, 
    \beta<\sup_{\xi\in\mathcal{D}_{f,\gamma}}\nu_s(\xi)\, 
    \Rightarrow\, \alpha+\beta<\mathcal{I}^*_{f,\gamma}(\nu).
  \end{equation}
  In the sequel, we fix an arbitrary $\nu=\nu_r+\nu_s\in ba$ and
  $\alpha, \beta$ as in (\ref{eq:ProofMainThKey1}), and we denote
  \begin{align*}
    \eta:=\frac{d\nu_r}{d\mathbb{P}} \quad\text{and}\quad \psi_Q
    :=\frac{dQ}{d\mathbb{P}},\, \forall Q\in\mathcal{Q}_\gamma.
  \end{align*}

  By the assumption on $\beta$, there exists a
  $\xi_s\in\mathcal{D}_{f,\gamma}$ with $\nu_s(\xi_s)>\beta$, and by
  the singularity of $\nu_s$, there exists an increasing sequence
  $(A_n)_n$ in $\mathcal{F}$ with $\mathbb{P}(A_n)\uparrow 1$ and
  $|\nu_s|(A_n)=0$ for all $n$, so that $\nu_s(\xi_s\ind_{A^c_n})
  =\nu_s(\xi_s)>\beta$.  On the other hand, the assumption on $\alpha$
  implies that
  \begin{align*}
    \alpha<\mathcal{H}_{f^*}(\nu_r|Q)+\gamma(Q)
    =\mathbb{E}\left[\tilde
      f^*\left(\cdot,\eta,\psi_Q\right)\right]+\gamma(Q), \quad\forall
    Q\in\mathcal{Q}_\gamma.
  \end{align*}
  Then for each $Q\in\mathcal{Q}_\gamma$, there exists a $\zeta_Q\in
  L^1$ with
  \begin{equation*}
    \mathbb{E}[\zeta_Q]>\alpha-\gamma(Q)\text{ and } 
    \zeta_Q<\tilde f^*(\cdot, \eta,\psi_Q)\text{ a.s.}
  \end{equation*}
  (even if $\Phi:=\tilde f^*\left(\cdot,\eta,\psi_Q\right)\not\in
  L^1$: since $\Phi^-\in L^1$, choosing $\varepsilon>0$ so that
  $\mathbb{E}[\Phi]-\varepsilon>\alpha$, we have
  $\lim_N\mathbb{E}[(\Phi-\varepsilon)\wedge N]>\alpha$ by the
  monotone convergence theorem, so $(\Phi-\varepsilon) \wedge N_0$
  with a big $N_0$ does the job.)  Therefore,
  Lemma~\ref{lem:MbleSelection1} implies that there exists a
  $\xi_Q^0\in L^0$ such that
  \begin{equation}
    \label{eq:ProofMainThKey3}
    f(\cdot, \xi_Q^0)<\infty\text{ and } \xi^0_Q\eta
    -\psi_Q f(\cdot, \xi^0_Q)\geq \zeta_Q\text{ a.s.}
  \end{equation}   

  Note that this does not guarantees that $\xi^0_Q$ is in
  $\mathcal{D}_{f,\gamma}$ (if it was, there would be nothing to prove
  anymore). So we approximate $\xi^0_Q$ by elements of
  $\mathcal{D}_{f,\gamma}$. Let $B_n:=\{|\xi^0_Q|\leq
  n\}\cap\{|f(\cdot, \xi^0_Q)|\leq n\}$ and $C_n:=A_n\cap B_n$, then
  $\mathbb{P}(C_n)\uparrow 1$ since $f(\cdot, \xi^0_Q)<\infty$
  a.s. Put
  \begin{align*}
    \xi^n_Q :=\xi^0_Q\ind_{C_n}+\xi_s\ind_{C^c_n},\quad\forall
    n\in\mathbb{N}.
  \end{align*}
  Then for each $n$, $\xi^n_Q\in\mathcal{D}_{f,\gamma}$ since
  $\|\xi^0_Q\|_\infty\leq n+\|\xi_s\|_\infty<\infty$ and $f(\cdot,
  \xi^n_Q)^+=f(\cdot, \xi^0_Q)^+\ind_{C_n}+f(\cdot, \xi_s)^+
  \ind_{C_n^c}\leq n +f(\cdot, \xi_s)^+\in M^{\rho_\gamma}_u$ by the
  solidness of the space.  On the other hand,
  \begin{align*}
    \nu_r(\xi^n_Q) -\mathbb{E}_Q\Bigl[f(\cdot, \xi^n_Q)\Bigr]
&    =\mathbb{E}\left[\xi^n_Q\eta-\psi_Q f(\cdot,
      \xi^n_Q)\right]\\
    &=\mathbb{E}\left[\ind_{C_n}\left(\xi^0_Q\eta-\psi_Qf(\cdot,
        \xi^0_Q)\right)\right]
    +\mathbb{E}\left[\ind_{C_n^c}\left(\xi_s\eta-\psi_Qf(\cdot,
        \xi_s)\right)\right]\\
    &\geq \mathbb{E}[\zeta_Q\ind_{C_n}]
    +\mathbb{E}\left[\ind_{C_n^c}\left(\xi_s\eta-\psi_Qf(\cdot,
        \xi_s)\right)\right]\\
    &=\mathbb{E}[\zeta_Q]
    +\mathbb{E}\Bigl[\ind_{C^c_n}\bigl(\underbrace{\xi_s\eta -\psi_Q
      f(\cdot, \xi_s)-\zeta_Q}_{\mathclap{=: \Xi_Q}}\bigr)\Bigr].
  \end{align*}
  Note that $\Xi_Q\in L^1$ since $f(\cdot, \xi_s)\in
  \bigcap_{Q\in\mathcal{Q}_\gamma}L^1(Q)$ (by
  $\xi_s\in\mathcal{D}_{f,\gamma}$), thus
  $\lim_n\mathbb{E}[\ind_{C_n^c}\Xi_Q]=0$. Therefore, noting that
  $\nu_s(\xi_Q^n) =\nu_s(\ind_{C_n^c}\xi_Q^n)
  =\nu_s(\ind_{C_n^c}\xi_s)=\nu_s(\xi_s)$,
  \begin{align*}
    \sup_{\xi\in\mathcal{D}_{f,\gamma}}
  \left(\nu(\xi)-\mathbb{E}_Q[f(\cdot,\xi)]\right) &\geq
    \sup_n\left(\nu_r(\xi^n_Q)-\mathbb{E}_Q[f(\cdot,\xi^n_Q)]+\nu_s(\xi_Q^n)\right)\\
    &\geq\limsup_n\left(\mathbb{E}[\zeta_Q]
      +\mathbb{E}[\ind_{C_n^c}\Xi_Q]+\nu_s(\xi_s)\right)
    \\
    & =\mathbb{E}[\zeta_Q]+\nu_s(\xi_s)>\alpha-\gamma(Q)+\beta.
  \end{align*}
  In view of (\ref{eq:Minimax1}), we have
  \begin{align*}
    \mathcal{I}_{f,\gamma}^*(\nu)
    &\stackrel{\text{(\ref{eq:Minimax1})}}\geq
    \inf_{Q\in\mathcal{Q}_\gamma} \sup_{\xi\in \mathcal{D}_{f,\gamma}}
    \left(\nu(\xi)-\mathbb{E}_Q[f(\cdot, \xi)]+\gamma(Q)\right)\geq
    \alpha+\beta.
  \end{align*}
  Since $\alpha<\mathcal{H}_{f^*,\gamma}(\nu_r)$ and
  $\beta<\sup_{\xi\in \mathcal{D}_{f,\gamma}}\nu_s(\xi)$ are
  arbitrary, this completes the proof.
\end{proof}

\subsection{Proof of Theorem~\ref{thm:LebMaxCompact}}
\label{sec:ProofLebesgue}

\emph{In the sequel, the assumptions of
  Theorem~\ref{thm:LebMaxCompact} \textbf{excepting}
  (\ref{eq:AddAssNegPart}) are supposed without notice}. The
implication (i) $\Rightarrow$ (ii) is trivial since $\mathbb{R}\subset
L^\infty$, and (i) $\Rightarrow$ (iii) is clear from (\ref{eq:EstConj})
of Theorem~\ref{thm:RobustRockafellar}, while (ii) $\Rightarrow$ (i)
follows from (cf. \citep{MR0310612,MR0512209})
\begin{equation}
  \label{eq:ConstAreEnough}
  a\leq \xi \leq b\text{ a.s., }a,b\in\mathbb{R}\,\Rightarrow\,    
  f(\cdot,\xi)\leq f(\cdot,a)^++f(\cdot,b)^+.
\end{equation}
Indeed, the assumption implies the existence of a $[0,1]$-valued
random variable $\alpha$ such that $\xi =\alpha a+(1-\alpha)b$
a.s. Since $f(\omega,\cdot)$ is convex for a.e. $\omega$, we see that
$f(\omega,\xi(\omega))\leq \alpha(\omega)f(\omega,a)
+(1-\alpha(\omega))f(\omega,b)\leq f(\omega,a)^+ +f(\omega,b)^+$ for
a.e. $\omega\in \Omega$.

Given $\mathrm{dom}(\mathcal{I}_{f,\gamma}) =L^\infty$ and
$\sigma(L^\infty,L^1)$-lsc of $\mathcal{I}_{f,\gamma}$
(Lemma~\ref{lem:WellDefFatou}) as well as
$\mathcal{I}_{f,\gamma}^*|_{L^1} =\mathcal{H}_{f^*,\gamma}$
(Corollary~\ref{cor:RobRockL1}), (iv) $\Leftrightarrow$ (v) is a
special case of the following (with $E=L^\infty$ and $E'=L^1$):
\begin{lemma}[\citep{MR0160093}, Propositions 1 and 2]
  \label{lem:Moreau}
  Let $\langle E,E'\rangle$ be a dual pair and $\varphi$ a
  $\sigma(E,E')$-lsc finite convex function on $E$ with the conjugate
  $\varphi^*$ on $E'$. Then $\varphi$ is $\tau(E,E')$-continuous on
  $E$ if and only if $\{x'\in E':\, \varphi^*(x')\leq c\}$ is
  $\sigma(E',E)$-compact for each $c\in\mathbb{R}$.
\end{lemma}

\begin{proof}[Proof of (iii) $\Rightarrow$ (iv)]
  Put $E=L^\infty$, $E'=ba$, then $\tau(L^\infty,ba)$ is the
  norm-topology while the \emph{finite} lsc convex function
  $\mathcal{I}_{f,\gamma}$ on the Banach space $L^\infty$ is
  norm-continuous (see \cite[Ch.1, Cor.~2.5]{ekeland_temam76}). Thus
  $\Lambda_c :=\{\nu\in ba:\, \mathcal{I}_{f,\gamma}(\nu)\leq c\}$ is
  $\sigma(ba,L^\infty)$-compact from Lemma~\ref{lem:Moreau}, a
  fortiori $\{\eta\in L^1:\, \mathcal{H}_{f^*,\gamma}(\eta)\leq
  c\}=\Lambda_c\cap L^1=\Lambda_c$ (by Corollary~\ref{cor:RobRockL1}
  and (iii)) is $\sigma(L^1,L^\infty)$-compact since
  $\sigma(L^1,L^\infty) =\sigma(ba,L^\infty)|_{L^1}$.
\end{proof}
\begin{proof}[Proof of (v) $\Rightarrow$ (vi)]
  This follows from the observation that
  \begin{align*}
    \sup_n\| \xi_n\|_\infty<\infty\text{ and } \xi_n\rightarrow
    \xi\text{ a.s. }\Rightarrow\, \xi_n\rightarrow \xi \text{ for }
    \tau(L^\infty,L^1).
  \end{align*}
  Indeed, for any weakly compact ($\Rightarrow$ uniformly integrable)
  subset $\mathcal{C}\subset L^1$,
  \begin{align*}
    \sup_{\eta\in \mathcal{C}} \mathbb{E}[|\xi-\xi_n||\eta|] &\leq
    \sup_{\eta\in \mathcal{C}}
    \mathbb{E}[|\xi-\xi_n||\eta|\ind_{\{|\eta|>N\}}] + \sup_{\eta\in
      \mathcal{C}}\mathbb{E}[|\xi-\xi_n||\eta|\ind_{\{|\eta|\leq N\}}]\\
    &\leq2 \sup_n\|\xi_n\|_\infty\sup_{\eta\in \mathcal{C}}
    \mathbb{E}[|\eta|\ind_{\{|\eta|>N\}}] +N\mathbb{E}[|\xi-\xi_n|].
  \end{align*} 
  Taking a diagonal, we see $q_\mathcal{C}(\xi-\xi_n):=\sup_{\eta\in
    \mathcal{C}}|\mathbb{E}[(\xi-\xi_n)\eta]|\rightarrow 0$, while
  $q_\mathcal{C}$ with $\mathcal{C}$ running through (convex, circled)
  weakly compact subsets of $L^1$ generates $\tau(L^\infty,L^1)$.
\end{proof}

\begin{proof}[Proof of (iv) $\Rightarrow$ (vii)]
  Fix an arbitrary $\eta_0\in\mathrm{dom}(\mathcal{H}_{f^*,\gamma})$
  ($\neq \emptyset$) and $\xi\in L^\infty$. Observe that
  \begin{align*}
    \mathbb{E}[\xi(\eta+\eta_0)]&= \mathbb{E}\left[2\xi\frac{\eta
        +\eta_0}2\right] \leq\mathcal{I}_{f,\gamma}(2\xi)
    +\frac12\left(\mathcal{H}_{f^*,\gamma}(\eta)
      +\mathcal{H}_{f^*,\gamma}(\eta_0)\right),\,\forall \eta\in L^1.
  \end{align*}
  Hence $\mathbb{E}[\xi\eta] -\mathcal{H}_{f^*,\gamma}(\eta)\leq
  \mathcal{I}_{f,\gamma}(2\xi) +\|\xi\|_\infty\|\eta_0\|_1
  +\frac12\mathcal{H}_{f^*,\gamma}(\eta_0)
  -\frac12\mathcal{H}_{f^*,\gamma}(\eta)$.  Putting $C_{c,\xi}
  :=2\left(c+\mathcal{I}_{f,\gamma}(2\xi)+\|\xi\|_\infty\|\eta_0\|_1\right)
  +\mathcal{H}_{f^*,\gamma}(\eta_0)$ which does not depend on $\eta$,
  we see that
  \begin{align*}
    \mathbb{E}[\xi\eta]-\mathcal{H}_{f^*,\gamma}(\eta)\geq
    c\,\Rightarrow\, \mathcal{H}_{f^*,\gamma}(\eta)\leq C_{c,\xi}.
  \end{align*}
  Consequently, $\{\eta\in L^1:\, \mathbb{E}[\xi\eta]
  -\mathcal{H}_{f^*,\gamma}(\eta)\geq c\}$ is weakly compact for each
  $c>0$, since it is contained in a weakly compact set $\{\eta\in
  L^1:\, \mathcal{H}_{f^*,\gamma}(\eta)\leq C_{c,\xi}\}$ and
  $\eta\mapsto \mathbb{E}[\xi\eta] -\mathcal{H}_{f^*,\gamma}(\eta)$ is
  weakly upper semicontinuous. Therefore, $\sup_{\eta\in L^1}
  \left(\mathbb{E}[\xi\eta]-\mathcal{H}_{f^*,\gamma}(\eta)\right)$ is
  attained.
\end{proof}

\emph{From now on, we assume (\ref{eq:AddAssNegPart})} (thus all the
assumptions of Theorem~\ref{thm:LebMaxCompact}) which implies that
\begin{equation}
  \label{eq:ConseqAddAss1}
  f(\cdot,\xi)^+\in M^{\rho_\gamma},\,\forall \xi\in L^\infty.
\end{equation}
Indeed, by $\mathrm{dom}(\mathcal{I}_{f,\gamma})=L^\infty$ and
$\exists \xi_0\in \mathcal{D}_{f,\gamma}$, we have for all $\xi\in
L^\infty$ that
\begin{align*}
  \rho_\gamma \left(\lambda f(\cdot, \xi)\right)&\leq\frac12
  \rho_\gamma \left(f(\cdot, 2\lambda \xi -(2\lambda-1)\xi_0)\right)
  +\frac12 \rho_\gamma\left((2\lambda-1)
    f(\cdot,\xi_0)^+\right)<\infty,\,\lambda>1.
\end{align*}
Here the first term in the right hand side is $\frac12
\mathcal{I}_{f,\gamma}(2\lambda \xi-(2\lambda-1)\xi_0)<\infty$. On the
other hand, if $f(\cdot,\xi_0')^-\in M^{\rho_\gamma}$ as in
(\ref{eq:AddAssNegPart}), putting $A=\{f(\cdot, \xi)\geq 0\}$ with an
arbitrary $\xi\in L^\infty$, we see that $f(\cdot,
\xi\ind_A+\xi'_0\ind_{A^c}) =f(\cdot, \xi)^++f(\cdot,
\xi'_0)\ind_{A^c}$, hence $f(\cdot,\xi)^+\leq f(\cdot,
\xi\ind_A+\xi'_0\ind_{A^c})+f(\cdot, \xi'_0)^-$. Therefore,
$\rho_\gamma\left(\lambda f(\cdot, \xi)^+\right) \leq \frac12
\rho_\gamma\left(2\lambda f(\cdot, \xi\ind_A+\xi'_0\ind_{A^c})\right)
+\frac12\rho_\gamma\left(2\lambda f(\cdot, \xi'_0)^-\right)$, $\forall
\lambda>1$.

\begin{proof}[Proof of (vi) $\Rightarrow$ (i)]
  In view of (\ref{eq:UIOrliczIFF}) and (\ref{eq:ConseqAddAss1}), it
  suffices to show that for each $\xi\in L^\infty$,
  \begin{equation*}
    \lim_N    \sup_{\gamma(Q)\leq 1}\mathbb{E}_Q\left[f(\cdot, \xi)^+\ind_{\{f(\cdot,
        \xi)^+\geq N\}}\right]= 0,\quad \forall c>0.
  \end{equation*}
  Let $A_N:=\{f(\cdot, \xi)\geq N\}\in\mathcal{F}$ ($N\in\mathbb{N}$),
  then $\mathbb{P}(A_N)\rightarrow 0$ since $\mathcal{I}_{f,\gamma}$
  is finite. Pick a $\xi_0\in \mathcal{D}_{f,\gamma}$ and put
  $\xi_N:=(\xi-\xi_0)\ind_{A_N}$ so that
  $\xi_N+\xi_0=\xi\ind_{A_N}+\xi_0\ind_{A_N^c}$. Then
  $f(\cdot,\xi_N+\xi_0)=f(\cdot, \xi)\ind_{A_N}+f(\cdot,
  \xi_0)\ind_{A_N^c}$,
  while for $\lambda>1$,
  $f(\cdot, \xi_N+\xi_0)\leq \frac1\lambda f(\cdot,
  \lambda\xi_N+\xi_0) +\frac{\lambda-1}\lambda f(\cdot, \xi_0)$, hence
  \begin{align*}
    f(\cdot, \xi)^+\ind_{A_N}&= f(\cdot, \xi)\ind_{A_N}\leq
    \frac1\lambda f(\cdot, \lambda\xi_N+\xi_0)+\frac1\lambda f(\cdot,
    \xi_0)^-+f(\cdot,\xi_0)^+\ind_{A_N},\quad\forall \lambda>1.
  \end{align*}

  Note that since $\xi_0\in \mathcal{D}_{f,\gamma}$ and since $f(\cdot,
  \xi_0)^-\in L^{\rho_\gamma}$,
  \begin{align*}
    C_1:= \sup_{\gamma(Q)\leq c}\mathbb{E}_Q[f(\cdot, \xi_0)^-]<\infty\text{
      and } \lim_N\sup_{\gamma(Q)\leq c}
    \mathbb{E}_Q[f(\cdot,\xi_0)^+\ind_{A_N}]=0.
  \end{align*}
  Also, $\mathbb{E}_Q[f(\cdot,
  \lambda\xi_N+\xi_0)]\leq\rho_\gamma(f(\cdot,\lambda\xi_N+\xi_0))+\gamma(Q)=
  \mathcal{I}_{f,\gamma}(\lambda\xi_N+\xi_0)+ \gamma(Q)$, while for each
  $\lambda>1$, we have $\sup_N\|\lambda\xi_N+\xi_0\|_\infty
  \leq\lambda\|\xi\|_\infty+(\lambda-1)\|\xi_0\|_\infty<\infty$ and
  $\lambda\xi_N+\xi_0\rightarrow \xi_0$ a.s., thus
  $\mathcal{I}_{f,\gamma}(\lambda\xi_N+\xi_0)\rightarrow
  \mathcal{I}_{f,\gamma}(\xi_0)$ by (vi), hence for some big $N_\lambda$,
  $\mathcal{I}_{f,\gamma}(\lambda\xi_N+\xi_0)\leq
  \mathcal{I}_{f,\gamma}(\xi_0)+1=:C_2$ for $N>N_\lambda$. Summing up,
  \begin{align*}
    \sup_{\gamma(Q)\leq c} \mathbb{E}_Q[f(\cdot, \xi)^+\ind_{A_N}]&\leq
    \frac{C_2+c +C_1}{\lambda} +\sup_{\gamma(Q)\leq
      c}\mathbb{E}_Q[f(\cdot,\xi_0)^+\ind_{A_N}],\,\forall N>N_\lambda,\,
    \lambda>1.
  \end{align*}
  Now a diagonal argument yields that $\sup_{\gamma(Q)\leq c}
  \mathbb{E}_Q[f(\cdot, \xi)^+\ind_{A_N}]\rightarrow 0$.
\end{proof}

\begin{proof}[Proof of (vii) $\Rightarrow$ (iv)]
  Under (\ref{eq:ConseqAddAss1}),
  $\lim_{\|\eta\|_1\rightarrow\infty}
  \mathcal{H}_{f^*,\gamma}(\eta)/\|\eta\|_1=\infty$,
  i.e., $\mathcal{H}_{f^*,\gamma}$ is \emph{coercive on $L^1$}.  For,
  since $\|\eta\|_1=\mathbb{E}[\mathrm{sgn}(\eta)\eta]$ where
  $\mathrm{sgn}(\eta)=\ind_{\{\eta\geq 0\}}-\ind_{\{\eta<0\}}\in
  L^\infty$,
  $f(\cdot, n\mathrm{sgn}(\eta)) \leq f(\cdot,-n)^++f(\cdot,n)^+\in
  M^{\rho_\gamma}$
  by (\ref{eq:ConstAreEnough}), (\ref{eq:ConseqAddAss1}), and
  $\mathcal{H}_{f^*,\gamma}(\eta)
  =\sup_{\xi\in
    L^\infty}(\mathbb{E}[\xi\eta]-\mathcal{I}_{f,\gamma}(\xi))$,
  we have
  \begin{align*}
    \mathcal{H}_{f^*,\gamma}(\eta)&\geq
    \mathbb{E}[n\mathrm{sgn}(\eta)\eta] -\mathcal{I}_{f,\gamma}(n\mathrm{sgn}(\eta))\\
    &\geq n\|\eta\|_1-\frac{\rho_\gamma\left(2f(\cdot,-n)^+\right)
      -\rho_\gamma\left(2f(\cdot, n)^+\right)}2>-\infty,\,\forall n.
  \end{align*}
  Then the \emph{coercive James' theorem}
  \citep[Th.~2]{MR2847443} applied to the coercive
  function $\mathcal{H}_{f^*,\gamma}$ on the Banach space $E=L^1$ implies the
  relative weak compactness of all the sublevel sets $\{\eta\in L^1:\,
  \mathcal{H}_{f^*,\gamma}(\eta)\leq c\}$ which are weakly closed since
  $\mathcal{H}_{f^*,\gamma}$ is weakly lower semicontinuous.
\end{proof}

\section*{Appendix}

\appendix

\section{Lower semicontinuity of $\mathcal{H}_{f^*}+\gamma$}
\label{sec:LSCEntropy}

\begin{lemma}
  \label{lem:DivergenceLSC}
  Under the assumptions of Theorem~\ref{thm:RobustRockafellar},
  $\mathcal{H}_{f^*}+\gamma$ is jointly weakly lower semicontinuous on
  $L^1\times\mathcal{Q}_\gamma$, and $\inf_{Q\in\mathcal{Q}_\gamma}
  \left(\mathcal{H}_{f^*}(\eta|Q)+\gamma(Q)\right)$ is attained for every
  $\eta\in L^1$.
\end{lemma}

\begin{proof}
  Let $\xi_0$ be as in (\ref{eq:IntegrabilityD}) (so
  $f(\cdot,\xi_0)^+\in M^{\rho_\gamma}_u$), $\psi_Q=dQ/d\mathbb{P}$ for each
  $Q\in\mathcal{Q}_\gamma$, and put
  $K(c,a):=2\left(c+\rho_\gamma\left(2f(\cdot,
      \xi_0)^+\right)+a\|\xi_0\|_\infty\right)<\infty$. Then by
  (\ref{eq:ProofMinimax1}) in the proof of
  Lemma~\ref{lem:ProofRockafellarMinimax1}, 
  \begin{align}
    \label{eq:ProofAppLSC1}
    \|\eta\|_1\leq a\text{ and } \mathcal{H}_{f^*}(\nu|Q)+\gamma(Q)\leq
    c\,\Rightarrow\, \gamma(Q)\leq K(c,a).
  \end{align}
  Also, from (\ref{eq:ProofPropEntropy1}) with $\xi_0$ above and
  (\ref{eq:UIOrliczIFF}), we see that whenever $(\eta_n)_n$ is
  uniformly integrable and $\sup_n\gamma(Q_n)<\infty$, $\bigl\{\tilde
  f^*(\cdot, \eta_n,\psi_{Q_n})^-\bigr\}_n$ is uniformly integrable.

  To see that $\mathcal{H}_{f^*}+\gamma$ is weakly lsc, it suffices by
  convexity that for each $c\in\mathbb{R}$,
  $\Lambda_c:=\{(\eta,\psi_Q)\in L^1\times L^1:\,
  Q\in\mathcal{Q}_\gamma,\, \mathcal{H}_{f^*}(\eta|Q)+\gamma(Q)\leq
  c\}$
  is norm-closed.  Let $(\eta_n,\psi_{Q_n})_n\subset \Lambda_c$ be
  norm convergent to $(\eta,\psi_Q)$.  Passing to a subsequence, we
  can assume a.s. convergence too.  By the above paragraph and the
  norm convergence, $(\eta_n)_n$ and
  $\bigl\{\tilde f^*(\cdot, \eta_n,\psi_{Q_n})^-\bigr\}_n$ are
  uniformly integrable, while $\gamma(Q)\leq \sup_n\gamma(Q_n)<\infty$
  by (\ref{eq:ProofAppLSC1}) and lsc of $\gamma$. Thus by Fatou's
  lemma,
  \begin{align*}
    \mathcal{H}_{f^*}(\eta|Q)+\gamma(Q) &=\mathbb{E}\left[\tilde
                                          f^*(\eta,\psi_Q)\right]+\gamma(Q)\\
                                        & \leq \liminf_n \mathbb{E}\left[\tilde
                                          f^*(\eta_n,\psi_{Q_n})\right]+\liminf_n\gamma(Q_n)\\
                                        &\leq \liminf_n \left( \mathbb{E}\left[\tilde
                                          f^*(\eta_n,\psi_{Q_n})\right]+\gamma(Q_n)\right)\leq c,
  \end{align*}
  hence $(\eta,\psi_Q)\in \Lambda_c$, obtaining the lower
  semicontinuity of $\mathcal{H}_{f^*}+\gamma$. In particular,
  $\mathcal{H}_{f^*}(\eta|\cdot)+\gamma(\cdot)$ is weakly lower semicontinuous
  on $\mathcal{Q}_\gamma$, so another application of (\ref{eq:ProofAppLSC1})
  as well as (\ref{eq:PenaltyFuncCompact}) imply that the infimum
  $\inf_{Q\in\mathcal{Q}_\gamma}\left(\mathcal{H}_{f^*}(\eta|Q)+\gamma(Q)\right)$ is
  attained for each $\eta\in L^1$.
\end{proof}

\section{Some Details of Example~\ref{ex:CounterEx1}}
\label{sec:DetailCounterEx}

Let $(\mathbb{N},2^\mathbb{N},\mathbb{P})$ and $(P_n)_n$ be as in
Example~\ref{ex:CounterEx1} with
$\mathcal{P}=\overline{\mathrm{conv}}(P_n, n\in\mathbb{N})$ (closed
convex hull). The weak compactness of $\mathcal{P}$ follows from
$\sup_nP_n(\{k,k+1,k+2,...\}) =\sup\{1/n:\, n\geq k\}=1/k\rightarrow
0$ as $k\rightarrow\infty$. Also $\rho_\gamma(\xi)
=\sup_{P\in\mathcal{P}}\mathbb{E}_P[\xi] =\sup_n\mathbb{E}_{P_n}[\xi]$
if $\xi\geq 0$ since $P\mapsto \mathbb{E}_P[\xi\wedge N]$ is
continuous for any $N\in \mathbb{N}$, so
$\sup_{P\in\mathcal{P}}\mathbb{E}_P[\xi]
=\sup_N\sup_{P\in\mathcal{P}}\mathbb{E}_P[\xi\wedge N]
=\sup_N\sup_{P\in
  \mathrm{conv}(P_n;n\in\mathbb{N})}\mathbb{E}_P[\xi\wedge N]$, while
if $P=\alpha_1 P_{n_1} +\cdots \alpha_l P_{n_l}$, then
$\mathbb{E}_P[\xi\wedge N] =\alpha_1E_{P_{n_1}}[\xi \wedge
N]+\cdots+\alpha_l E_{P_{n_l}}[\xi\wedge N] \leq \max_{1\leq i\leq l}
E_{P_{n_i}}[\xi\wedge N]\leq \sup_nE_{P_n}[\xi\wedge N]$.

\begin{lemma}
  \label{lem:Ex4LimSup} 
  Let $f$ be given by (\ref{eq:Ex4F}) in
  Example~\ref{ex:CounterEx1}. Then we have
  \begin{equation}
    \label{eq:ClaimLimSup2}
    \lim_{N\rightarrow\infty} 
    \sup_n \mathbb{E}_{P_n}[f(\cdot,\xi)\ind_{\{f(\cdot,\xi)\geq N\}}] 
    =\limsup_n \xi(n)^+e^{\xi(n)^+}.
  \end{equation}
\end{lemma}
\begin{proof}
  Let $h(x):=x^+e^x$ and fix $\xi\in \ell^\infty$. If
  $\|\xi\|_\infty=0$, then both sides of (\ref{eq:ClaimLimSup2}) are
  $0$, thus we assume $\|\xi\|_\infty>0$ ($\Leftrightarrow$
  $h(\|\xi\|_\infty)>0$).  Note that for any $N, n\in\mathbb{N}$, we have (by
  definition)
  \begin{equation*}
    \label{eq:ProofNXEX2}
    \begin{split}
      \mathbb{E}_{P_n}[f(\cdot,\xi)\ind_{\{f(\cdot,\xi)\geq N\}}]
      &=\left(1-\frac1n\right)h(\xi(1))\ind_{\{h(\xi(1))\geq N\}}
      +h(\xi(n))\ind_{\{nh(\xi(n))\geq N\}}.
    \end{split}
  \end{equation*}
  In particular, $\mathbb{E}_{P_n}[f(\cdot,\xi)\ind_{\{f(\cdot,\xi)\geq N\}}]
  \leq h(\xi(1))\ind_{\{h(\xi(1))\geq N\}}+h(\xi(n))\ind_{\{n
    h(\|\xi\|_\infty)\geq N\}}$, thus
  \begin{align*}
    \lim_{N\rightarrow\infty}
    \sup_n\mathbb{E}_{P_n}[f(\cdot,\xi)\ind_{\{f(\cdot,\xi) \geq N\}}]
    &\leq \lim_{N\rightarrow \infty} \sup_{n\geq N/h(\|\xi\|_\infty)}
    h(\xi(n))=\limsup_nh(\xi(n))=:\alpha.
  \end{align*}
  This is ``$\leq$'' in (\ref{eq:ClaimLimSup2}). If $\alpha=0$, we are
  done; otherwise, for any $\varepsilon>0$ and $N\in \mathbb{N}$, there
  exists some $n_N^\varepsilon>N/(\alpha-\varepsilon)>0$ with
  $h(\xi(n_N^\varepsilon))>\alpha-\varepsilon$.  In particular,
  $n_N^\varepsilon h(\xi(n_N^\varepsilon))>N$, hence
  \begin{align*}
    \sup_n\mathbb{E}_{P_n}[f(\cdot,\xi)\ind_{\{f(\cdot,\xi)\geq N\}}]
    &\geq \sup_n h(\xi(n))\ind_{\{nh(\xi(n))\geq N\}}\geq
    h(\xi(n_N^\varepsilon))>\alpha-\varepsilon.
  \end{align*}
  This proves ``$\geq$'' in (\ref{eq:ClaimLimSup2}).
\end{proof}

Since $h(x)=x^+e^x$ is increasing, continuous and $h(0)=0$, $\limsup_n
h(\xi(n))=0$ $\Leftrightarrow$ \linebreak$\limsup_n \xi(n)=0$, hence
(\ref{eq:CounterExDfgamma}). Recall that
$\mathrm{dom}(\mathcal{I}_{f,\gamma}^*)\subset ba_+$ (since $\mathcal{I}_{f,\gamma}$ is
increasing).  Finally, 
\begin{lemma}
  \label{lem:CountExConj1}
  The conjugate $\mathcal{I}_{f,\gamma}^*$ is explicitly given on $ba^s_+$ as:
  \begin{align}
    \label{eq:NotDNotDom}
    \mathcal{I}_{f,\gamma}^*(\nu)&= \sup_{x\geq 0} x(\|\nu\|-e^x),\quad\forall
    \nu\in ba^s_+.
  \end{align}  
\end{lemma}

\begin{proof}
  Let $\nu\in ba_+^s$.  Since $\mathbb{E}_{P_n}[f(\cdot,\xi)]
  =\left(1-\frac{1}{n}\right)h(\xi(1))+h(\xi(n))$, we have
  \begin{align*}\label{eq:ProofNXEX3}
    h(\xi(n))\stackrel{\text{($*$)}}\leq
    \mathbb{E}_{P_n}[f(\cdot,\xi)]\stackrel{\text{($**$)}}\leq
    h(\xi(1))+h(\xi(n)).
  \end{align*}
  From ($*$) and (\ref{eq:baSmallLinftySingular}),
  $\nu(\xi)-\mathcal{I}_f(\xi) \leq \|\nu\|\limsup_n \xi(n)-\sup_nh(\xi(n))
  \leq \|\xi^+\|_\infty (\|\nu\|-e^{\|\xi^+\|_\infty})\leq \sup_{x\geq
    0}x(\|\nu\|-e^x)$ which shows ``$\leq$'' in (\ref{eq:NotDNotDom}).
  Considering $\bar x^0:=(0,x,x,\cdots)\in \ell^\infty$ with $x\geq 0$
  (then $\|\bar x^0\|_\infty=x$ and $\nu(\bar x^0)=x\|\nu\|$ since
  $\nu$ vanishes on finite sets), we deduce from ($**$) that
  \begin{align*}
    \mathcal{I}_{f,\gamma}^*(\nu) &
                                    \geq \sup_{\xi\in
                                    \ell^\infty} (\nu(\xi)-h(\xi(1))- h(\|\xi\|_\infty)) \geq\sup_{x
                                    \geq0}(x\|\nu\|-xe^x).
  \end{align*}
\end{proof}

\section*{Acknowledgements}
\addcontentsline{toc}{section}{Acknowledgements}
The author gratefully acknowledges the financial support of the Center
for Advanced Research in Finance (CARF) at the Graduate School of
Economics of the University of Tokyo.

\end{document}